\documentclass[a4paper]{amsart}
\pdfoutput=1

\usepackage{
	amsmath,amssymb, 
	graphicx, 
	mathtools,
	bm,
	todonotes,
	hypbmsec,
	float,
	subcaption,
	caption,
	adjustbox
}
\usepackage[shortlabels]{enumitem}

\RequirePackage[colorlinks,citecolor=blue,urlcolor=blue]{hyperref}
\usepackage{apacite}
\usepackage[utf8]{inputenc}
\usepackage[margin=3cm]{geometry}

\usepackage{etoolbox}

\usepackage[
round,
]{natbib}

\usepackage[T1]{fontenc}
\usepackage[utf8]{inputenc}

\usepackage{tikz,pgfplots}
\usepgfplotslibrary{fillbetween}
\usepgfplotslibrary{colormaps}
\usetikzlibrary{calc}
\pgfplotsset{compat=1.17}

\definecolor{AUblue}{cmyk}{1,0.8,0,0.15}
\definecolor{AUorange}{cmyk}{0,0.6,1,0}
\definecolor{BrickRed}{HTML}{B6321C}
\definecolor{BurntOrange}{HTML}{F7921D}
\definecolor{ForestGreen}{HTML}{009B55}
\definecolor{drawColor}{gray}{0.20}

\colorlet{nodecolor}{AUblue!10}
\colorlet{col1}{AUblue!20}
\colorlet{col2}{AUblue!40}
\colorlet{col3}{AUblue!60}
\colorlet{col4}{AUblue!80}
\colorlet{col5}{AUblue!100}

\newcommand{\indep}{\perp}
\newcommand{\indepABC}{A \indep B \mid C\;[\Lambda]}
\newcommand{\indepAB}{A \indep B\;[\Lambda]}
\newcommand{\Perp}{\perp \!\!\! \perp}

\DeclarePairedDelimiterX\abs[1]\lvert\rvert{
	\ifblank{#1}{\:\cdot\:}{#1}
}
\DeclarePairedDelimiterX\norm[1]\lVert\rVert{ 
	\ifblank{#1}{\:\cdot\:}{#1}
}
\DeclarePairedDelimiterX{\inner}[2]{\langle}{\rangle}{ 
	\ifblank{#1}{\:\cdot\:}{#1},\ifblank{#2}{\:\cdot\:}{#2}
}

\providecommand\given{}
\DeclarePairedDelimiterX\Set[1]{\lbrace}{\rbrace}{
	\renewcommand\given{\SetSymbol[\delimsize]}
	#1
}
\DeclarePairedDelimiterXPP\Prob[1]{\mathbb{P}}(){}{
	\renewcommand\given{\nonscript\:\delimsize\vert\nonscript\:
		\mathopen{}}
	#1}
\DeclarePairedDelimiterXPP\Var[1]{\mathrm{Var}}(){}{
	\renewcommand\given{\nonscript\:\delimsize\vert\nonscript\:
		\mathopen{}}
	#1}
\DeclarePairedDelimiterXPP\Cov[2]{\mathrm{Cov}}(){}{ 
	\ifblank{#1}{\:\cdot\:}{#1},\ifblank{#2}{\:\cdot\:}{#2}
}
\DeclarePairedDelimiterXPP\Mean[1]{\mathbb{E}}[]{}{
	\renewcommand\given{\nonscript\:\delimsize\vert\nonscript\:
		\mathopen{}}
	#1}

\DeclareMathOperator{\MST}{mst}
\DeclareMathOperator{\argmin}{arg\,min}

\theoremstyle{plain}
\newtheorem{theorem}{Theorem}[section]

\newtheorem{definition}[theorem]{Definition}

\newtheorem{proposition}[theorem]{Proposition}
\newtheorem{lemma}[theorem]{Lemma}

\theoremstyle{remark}
\newtheorem{remark}[theorem]{Remark}
\newtheorem{example}[theorem]{Example}

\newcommand\N{\mathbb{N}}

\newcommand\R{\mathbb{R}}

\newcommand{\diag}{{\rm diag}}

\usepackage[mathscr]{euscript} 

\DeclarePairedDelimiterX\lrangle[1]\langle\rangle{
	\ifblank{#1}{\:\cdot\:}{#1}
}

\newcommand{\ep}{{\epsilon}}

\newcommand\idd{\,\mathrm{d}} 
\newcommand\dd{\mathrm{d}} 
\newcommand\ff{\mathcal{F}}

\newcommand\B{\mathcal{B}} 
\newcommand\E{\mathcal{E}}

\newcommand\G{\mathcal{G}}

\newcommand\HR{H\"usler--Reiss}

\newcommand{\p}{\mathbb P}
\newcommand{\e}{\mathbb E}
\renewcommand{\a}{\alpha}

\newcommand{\ind}[1]{\mbox{\rm\large 1}_{{#1}}}
\newcommand{\phT}{{\mathrm{ph}_T}}

\NewDocumentCommand\convd{s}{%
  \IfBooleanTF{#1}{
    \xrightarrow{\scriptscriptstyle\smash{d}}
  }{
    \xrightarrow{\ d\ }
  }
}

\NewDocumentCommand\cip{s}{%
  \IfBooleanTF{#1}{
    \xrightarrow{\scriptscriptstyle\smash{\p}}
  }{
    \xrightarrow{\ \p\ }
  }
}

\NewDocumentCommand\vaguely{s}{%
  \IfBooleanTF{#1}{
    \xrightarrow{\scriptscriptstyle\smash{v}}
  }{
    \xrightarrow{\ v\ }
  }
}

\ExplSyntaxOn
\NewDocumentCommand\ph{m}{
  \tl_set:Nn \l_tmpa_tl { #1 }
  \regex_replace_once:nnN { \A\s*\(\s* } {  } \l_tmpa_tl
  \regex_replace_once:nnN { \s*\)\s*\z } {  } \l_tmpa_tl 
  \sp{\smash{(}{\tl_use:N \l_tmpa_tl}\smash{)}}
}
\ExplSyntaxOff

\newcommand\indnew[1]{\operatorname{\mathbf{1}}{#1}}
\newcommand\deq{\stackrel{\smash{\scriptscriptstyle d}}{=}}

\NewDocumentCommand\For{sm}{
  \IfBooleanF{#1}{\quad}
  \quad\text{#2}
}
\NewDocumentCommand\qtq{sm}{
  \IfBooleanTF{#1}{\enspace}{\quad}
  \text{#2}
  \IfBooleanTF{#1}{\enspace}{\quad}
}

\title{L\'evy graphical models}
\author{Sebastian Engelke}
\address{Research Institute for Statistics and Information Science, University of Geneva, Geneva, Switzerland}
\author{Jevgenijs Ivanovs}
\address{Department of Mathematics, Aarhus University, Aarhus, Denmark}
\author{Jakob D. Th{\o}stesen}
\address{Department of Mathematics, Aarhus University, Aarhus, Denmark}

\begin{document}

\begin{abstract}
	Conditional independence and graphical models are crucial concepts for sparsity and statistical modeling in higher dimensions. For L\'evy processes, a widely applied class of stochastic processes, these notions have not been studied. By the L\'evy--Itô decomposition, a multivariate L\'evy process can be decomposed into the sum of a Brownian motion part and an independent jump process. We show that conditional independence statements between the marginal processes can be studied separately for these two parts. While the Brownian part is well-understood, we derive a novel characterization of conditional independence between the sample paths of the jump process in terms of the L\'evy measure. We define L\'evy graphical models as L\'evy processes that satisfy undirected or directed Markov properties. We prove that the graph structure is invariant under changes of the univariate marginal processes. L\'evy graphical models allow the construction of flexible, sparse dependence models for L\'evy processes in large dimensions, which are interpretable thanks to the underlying graph. For trees, we develop statistical methodology to learn the underlying structure from low- or high-frequency observations of the L\'evy process and show consistent graph recovery. We apply our method to model stock returns from U.S.~companies to illustrate the advantages of our approach.
\end{abstract}

\maketitle

\section{Introduction}

A L\'evy process is an $\mathbb R^d$-valued stochastic process $\mathbf{X} = (X(t))_{t \geq 0}$ with 
independent and stationary increments that satisfies $X(0) = 0$ 
almost surely. For any fixed $t\geq 0$, the univariate marginal distribution of 
$X(t)$ is infinitely divisible and, thus, it arises as the limit of row-sums of 
a triangular array \citep[Chapter 7]{kallenberg3}. For this reason, L\'evy processes are fundamental objects studied
in limit theory and they arise naturally in numerous settings. 
The L\'evy--Itô decomposition states that any L\'evy process can be decomposed
into two independent processes,
\[ \mathbf X = \mathbf W + \mathbf J,\] 
where the continuous part $\mathbf{W} = (W(t))_{t \geq 0}$ is a Brownian motion with covariance $\Sigma$ and drift $\gamma \in \mathbb R$, and the jump process $\mathbf{J} = (J(t))_{t \geq 0}$ is described by the so-called L\'evy measure $\Lambda$.

In the last decades, the probabilistic properties of L\'evy processes have been extensively studied, ranging from detailed analysis of the sample path behavior \citep{blu1961}, over stochastic differential equations \citep{Fasen_2005, bol2019}, to applications in physics~\citep{woy2001}, finance \citep{tankov2003financial} and other areas \citep{krz2015}. 
Surprisingly, there is little known about conditional independence in this class of stochastic processes. In probability theory and statistics, conditional independence is a crucial notion of irrelevance that is at the heart of many fields such as graphical models and causality \citep[e.g.,][]{lauritzen96, wainwright2008}. It further allows the definition of sparsity and is therefore the backbone for modern theory and statistical methods in larger dimensions.

One reason for the lacking connection is that the seemingly most natural definition of conditional independence for a L\'evy process does not lead to useful characterizations. Indeed, for disjoint subsets $A,B,C \subseteq \{1,\dots , d\}$ and a fixed time point $t>0$, one might be tempted to consider the conditional independence 
$X_A(t) \Perp X_B(t) \mid X_C(t)$.
While such a statement is of course well-defined, we will argue that it is neither natural nor useful for theory or practice. Intuitively, the issue is that conditioning only on the vector
$X_C(t)$ collapses important information on the stochastic process up to time $t$ into a single vector. 

Instead, we propose to study conditional independence on the level of sample paths. As a first fundamental result, we show the (non-trivial) equivalence 
\begin{equation}\label{CI_decomp}
	\mathbf{X}_A\Perp\mathbf{X}_B\mid\mathbf{X}_C\qtq{$\Leftrightarrow$}\mathbf{W}_A\Perp\mathbf{W}_B\mid\mathbf{W}_C\qtq*{and}\mathbf{J}_A\Perp\mathbf{J}_B\mid\mathbf{J}_C,
\end{equation}	
that is, the characterization of conditional independence of the process $\mathbf{X}$ can be separated into the corresponding statements for the Brownian and jump parts.
Conditional independence for the Brownian motion part is well understood: if the corresponding covariance matrix $\Sigma$ is invertible, then any conditional independence statement can be read off from the precision matrix $\Sigma^{-1}$.

Our main result characterizes conditional independence of the jump part $\mathbf J$. Since this part of the L\'evy process is described by the L\'evy measure $\Lambda$, we would like to describe stochastic properties in terms of this object. In fact, it turns out that exactly this is possible for conditional independence statements. Under a mild condition on $\Lambda$ we show that 
for the jump process we have
\begin{equation}\label{CI_jump}
	\mathbf{J}_A\Perp\mathbf{J}_B\mid\mathbf{J}_C\qtq{$\Leftrightarrow$} A\perp B\mid C\,[\Lambda],
\end{equation}
where the conditional independence notion on the right-hand side is a natural generalization of classical conditional independence to infinite measures exploding at the origin \citep{eng_iva_kir}.
This provides an effective tool to study conditional independence at the stochastic process level, to construct concrete examples and develop statistical methodology for sparse L\'evy processes by considering the simpler object $\Lambda$.

Our theory allows us to define graphical models for L\'evy processes, a mostly unexplored field with many open questions and large potential for statistical applications. 
For an undirected graph $G=(V,E)$ with nodes $V$ and edges $E$, a L\'evy graphical model is defined by the global Markov property \citep{lauritzen96}, which implies that whenever there is no edge between two nodes they are conditionally independent given the remaining nodes of $G$ in the sense of~\eqref{CI_jump}. The left-hand side of Figure~\ref{fig:3D_path} shows a realization of a 3-dimensional L\'evy process $\mathbf{X}$. Clearly, neither of the three components are independent since the paths partially exhibit simultaneous jumps. While not obvious from visual inspection of the plot, the processes $\mathbf{X}_1$ and $\mathbf{X}_3$ are conditionally independent given $\mathbf{X}_2$.
This example can be generalized to construct flexible, sparse dependence models for stochastic processes in large dimensions. Indeed, it suffices to specify the lower-dimensional interactions on smaller parts of the graph and the conditional independence structure implies a full $d$-dimensional model.

A popular approach for parametric modelling of the dependence structure of a multi-dimensional L\'evy processes is the concept of L\'evy copulas introduced in \cite{tankov2003financial} and \cite{kallsen_tankov}. Our L\'evy graphical models are complementary to these methods and offer several advantages. They are structural approaches allowing for interpretation of the graph structure~$G$, efficient computations for statistical inference and constructions in high dimensions by relying on lower-dimensional sub-models. 
The two approaches are compatible since we show that the conditional independence statements in~\eqref{CI_jump} are invariant under the behavior of the marginal, univariate processes and only depend on the L\'evy copula.

\begin{figure}[tb]
\centering
\includegraphics[width=0.48\textwidth]{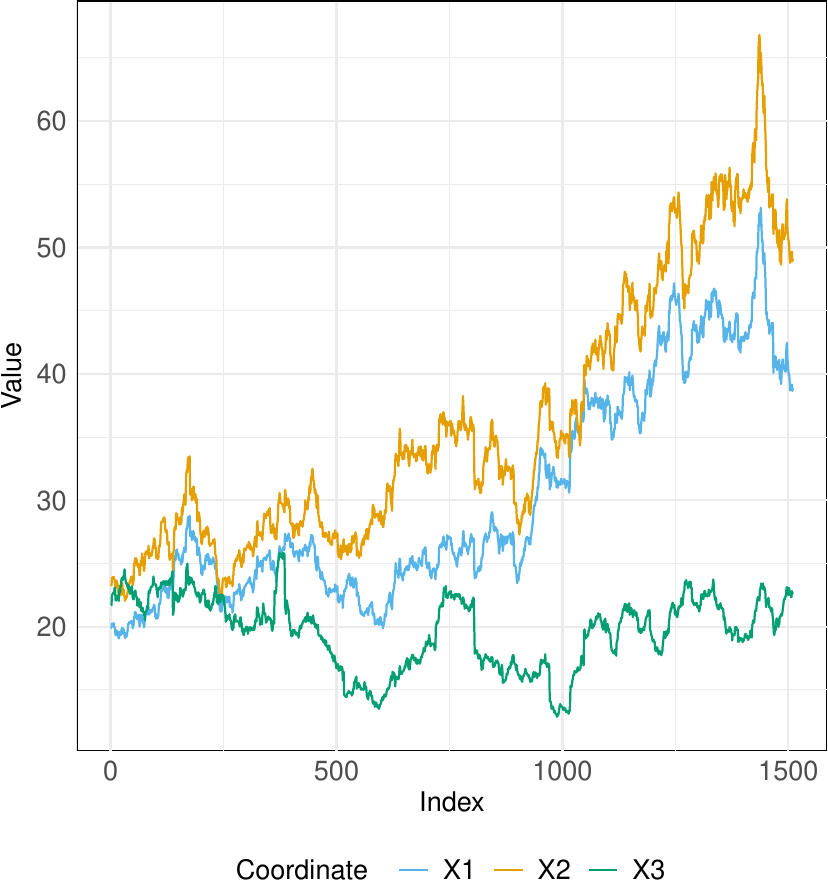}\hspace*{1em}
\includegraphics[width=0.48\textwidth]{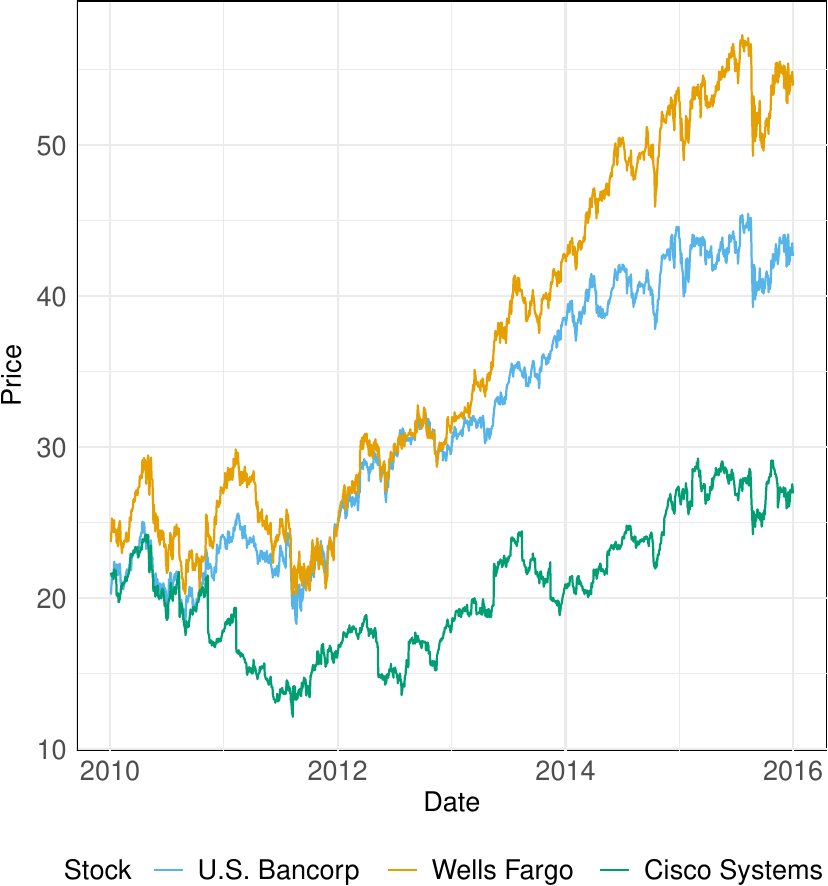} 
\caption{Left: Realization of a 3-dimensional L\'evy graphical model $\mathbf X$ on graph $G = (\{1,2,3\}, \{(1,2), (2,3)\})$ with conditional independence $\mathbf{X}_1\Perp\mathbf{X}_3\mid\mathbf{X}_2$. Right: Daily stock prices of three U.S.~companies used in the data application.}
\label{fig:3D_path}
\end{figure}

While the notion of L\'evy graphical models is applicable to general graph structures $G$, in this paper we concentrate on model construction and statistical inference for the class of trees, that is, connected graphs without cycles. Relying only on bivariate interactions, L\'evy tree models highlight the power of our approach in constructing simple yet realistic dependence models for applications. In particular, we show that the underlying tree structure $T$ can be estimated non-parametrically from the data in both classical asymptotic and high-frequency settings. We prove consistency of this tree estimator $\widehat T$ for a large class of L\'evy processes as the number of observed time points goes to infinity.

We apply our methodology to model dependence between the stock price processes of $d=16$ U.S.~companies. The right-hand side of Figure~\ref{fig:3D_path} shows the daily observations of three of those price processes. Our analysis of this data and the estimated tree on these 16 companies yield, for instance, the conditional independence between the paths of U.S.~Bancorp and Cisco Systems given the information of the sample path of Wells Fargo. The simulations on the left-hand side of this figure are samples of our fitted model. They seem to capture well the dependence features of the real data and can therefore be used for stress testing, option pricing or other financial downstream tasks \citep[e.g.,][Chapter III]{tankov2003financial}.

We first review some existing notions of dependence for stochastic processes in Section~\ref{review}.
Section~\ref{sec:prelim} provides necessary background on L\'evy processes and discusses the notions of standardized L\'evy measures and L\'evy copulas. The first fundamental results on conditional independence for L\'evy processes are derived in Section~\ref{levy_CI}. We also study the difference of conditioning on the process at a fixed time point or on the whole sample path. Section~\ref{sec:jump_part} defines a suitable notion of conditional independence for the L\'evy measure $\Lambda$ and links this to the conditional independence of the sample paths of the jump process $\mathbf J$ in~\eqref{CI_decomp}. This conditional independence is shown to be invariant under changes of the marginal processes. In Section~\ref{sec:levy_GM} we define directed and undirected L\'evy graphical models and propose a tractable model on tree structures. The statistical inference, including consistent tree recovery based on low- and high-frequency observations of the L\'evy process, is presented in Section~\ref{sec:structure_learning_trees}. Section~\ref{sec:experiments} contains a simulation study to assess the performance of our tree learning algorithm and an application of L\'evy graphical models to stock returns from U.S. companies.

The proofs of all results can be found in the appendix.
The data and R code to reproduce the results and figures of this paper are available at \url{https://github.com/sebastian-engelke/levy_graphical_models}.

\subsection{Dependence models for stochastic processes}\label{review}

Modeling dependence between different components of a multi-dimensional stochastic process is a natural problem and has been studied in many settings. We give a short overview over some of the approaches with a focus on concepts that aim for structural models with interpretable and intuitive representations. 

A popular concept is Granger causality introduced in \citet{granger_69}. The general idea is to study how information about one component of a time series up to time $n-1$ influences prediction of the value of another component at time $n$. Granger causality has been very successful and led to other works such as \citet{eichler_12}, who introduce graphical models for multivariate time series. A generalization of Granger causality to continuous time is known as local independence, which describes how information of one coordinate influences the infinitesimal increment of another coordinate. Local independence is considered for various classes of stochastic processes. For instance, \citet{didelez_08} define graphical models for marked point processes and \citet{mog_han_22} study diffusion models with correlated driving noise.

Granger causality and local independence are both based on dependence in time. These concepts therefore lead to trivial statements for L\'evy processes, which have independent and stationary increments. Indeed, the evolution of a L\'evy process after a fixed time point $T>0$ is independent of the process up to time $T$. In this paper we therefore concentrate instead on dependence in space, that is, the dependence structure of jumps in the different coordinates occurring at the same time. 

The notion of $\alpha$-stable graphical models is introduced in \citet{mis_kur_16} and \cite{muvunza2024cauchy}. These models describe a construction of an $\alpha$-stable random vector on a directed acyclic graph. We show that the associated L\'evy processes are indeed a special case of our L\'evy graphical models. However, only spectrally discrete processes appear through this construction and the corresponding processes are fairly basic; for instance, the L\'evy measure cannot admit Lebesgue densities. Our general theory allows for much broader classes of models.

To model dependence of the jumps of a L\'evy process \citet{kallsen_tankov} use L\'evy copulas, an extension of copulas for random vectors to the level of stochastic processes. A somewhat similar approach are the so-called Pareto L\'evy measures in \citet{eder_klup}. Both works choose to model the dependence via the L\'evy measure, which links it to the approach presented in our paper. The concepts of copulas and graphical models are however fundamentally different: the former aims at constructing parametric models in typically low dimensions, whereas the latter is a more structured approach for interpretable and sparse models in possibly high dimensions.

\section{Preliminaries}\label{sec:prelim}

This section provides a brief introduction to L\'evy processes and the concept of standardized L\'evy measures.
Throughout we denote the index set by $V = \Set{1,\dotsc,d}$, which will later be the set of vertices in our graphical models. We often abbreviate sets of points $x\in\R^d$ and write, for instance, $\Set{x\in A}=\Set{x\in\R^d : x\in A}$ and~$\Set{x_1\geq1} = \Set{x\in\R^d : x_1\geq1}$. We denote stochastic processes $\mathbf X = (X(t))_{t \geq 0}$ in bold face, and random variables and vectors such as $X(1)$, the process at time 1, in normal text. For an index set $I\subset V$, we write $y_I = (y_i)_{i\in I}$ for subvectors of $y\in \mathbb R^d$. The Borel sets of $\mathbb R^d$ are denoted by $\mathcal B(\R^d)$.

\subsection{L\'evy process}
\label{sec:levy_def}

Let $\mathbf{X}=(X(t))_{t\geq0}$ be a stochastic process defined on some probability space $(\Omega,\ff,\p)$, and taking values in $\R^d$. 
The space of càdlàg functions is defined as $\mathbb D = \{f:[0,\infty)\to \mathbb R^d: f \text{ is right continuous with left limits} \}$ and is a Polish space when endowed with the Skorokhod topology \citep[][Theorem~1.14]{jacod_shiryaev}. We say that $\mathbf{X}$ is a L\'evy process if it satisfies the following conditions:
\begin{enumerate}[(i)]
\item $\p(X(0)=0)=1$.
\item For $n\in\N$ and $0\leq t_0<t_1<\dotsm<t_n$ the increments $X(t_1)-X(t_0),\dotsc,X(t_n)-X(t_{n-1})$ are independent.
\item For $s,t\geq 0$, the distribution of $X(t+s)-X(s)$ does not depend on $s$.
\item The sample path $t\mapsto X(t)$ is càdlàg $\p$-almost surely.
\end{enumerate}
L\'evy processes have been extensively studied and we refer to \citet{bertoin96}, \citet{sato99} and \citet{applebaum} for some classical reference books.

An important consequence of the above definition is that the distribution of $X(1)$ is infinitely divisible. Its characteristic function is therefore given by the L\'evy--Khintchine formula \citep[e.g.,][Theorem~8.1]{sato99}
\begin{equation*}
\Mean{e^{i\inner{u}{X(1)}}}=\exp\left(i\inner{u}{\gamma}-\frac{1}{2}\inner{u}{\Sigma u}+\int_{\R^d}e^{i\inner{u}{x}}-1-i\inner{u}{x}\ind{\Set{\norm{x}\leq1}}\,\Lambda(\dd x)\right),\quad u\in\R^d,
\end{equation*}
where $\gamma\in\R^d$, $\Sigma$ is a positive semidefinite $d\times d$ matrix, $\norm{\cdot}$ is a norm on $\R^d$, and the so-called L\'evy measure $\Lambda$ is a measure on $\R^d$ and satisfies
\begin{equation}\label{eq:integrability_cond}
\int_{\R^d}(1\wedge\norm{x}^2)\,\Lambda(\dd x)<\infty.
\end{equation}
In particular, $\Lambda(A)<\infty$ for any Borel set $A\in\B(\R^d)$ bounded away from $0$. On the other hand, the L\'evy measure can explode at the origin, in which case $\Lambda(\R^d) = \infty$ and $\mathbf X$ has infinitely many jumps before any time point $t>0$. 
We say that $\mathbf{X}$ has characteristic triplet $(\gamma,\Sigma,\Lambda)$. 
 While any L\'evy process gives rise to an infinitely divisible distribution, the converse is also true. Indeed, for any infinitely divisible distribution $\nu$ on $\R^d$ there exists a L\'evy process $\mathbf{X}$ such that~$X(1)\sim\nu$. 

The L\'evy--Itô decomposition provides a stochastic representation of the L\'evy process $\mathbf X$ with characteristic triplet $(\gamma,\Sigma,\Lambda)$ as
\[\mathbf{X}=\mathbf{W}+\mathbf{J},\]
where $\mathbf{J}$ is a L\'evy process with characteristic triplet $(0,0,\Lambda)$, and $\mathbf{W}$ is a Brownian motion with drift $\gamma$ and covariance matrix $\Sigma$, which is independent of $\mathbf{J}$. Since $\mathbf{W}$ is almost surely everywhere continuous, it is common to refer to $\mathbf{W}$ and~$\mathbf{J}$ as the Brownian part and the jump part of $\mathbf{X}$, respectively. Importantly, the terms $\mathbf{W}$ and~$\mathbf{J}$ may be constructed as almost sure limits of certain processes $\mathbf{W}\ph{(n)}$ and $\mathbf{J}\ph{(n)}$ created from $\mathbf{X}$ \citep[e.g.,][\S2.4]{applebaum}; see Appendix~\ref{sec:decomp_cond} for details. This fact will be used to prove the result~\eqref{CI_decomp} later.

We give some typical examples of $d$-dimensional L\'evy processes; see also Figure~\ref{fig:Levy_example} for an illustration in one dimension $d=1$.
\begin{example}\label{ex:BM}
	A L\'evy process $\mathbf X$ with characteristic triplet $(\gamma,\Sigma,0)$ can be written as 
	\[ X(t) = \gamma t + B(t), \quad t \geq 0,\]
	where $\mathbf B$ is a zero-mean Brownian motion with covariance matrix of $W(1)$ given by $\Sigma$.
\end{example}

\begin{example}
	Suppose that the L\'evy measure satisfies $\Lambda(\R^d) < \infty$. We say that $\mathbf X$ is a compound Poisson process with rate $\eta = \Lambda(\R^d)$ and jump distribution $\mu(\cdot) = \Lambda(\cdot) / \Lambda(\R^d)$ if the characteristic triplet is $(\gamma,0,\Lambda)$ with $\gamma=\int_{\Set{\norm{x}\leq1}}x\,\Lambda(\dd x)$.
\end{example}

\begin{example}\label{ex:stable}
	A L\'evy process $\mathbf{X}$ is said to be $\alpha$-stable with stability index $\alpha \in(0,2]$, if for any $h>0$ there exists $c\in\R^d$ such that
	\begin{equation}\label{eq:stable}
		X(h)\deq h^{1/\alpha} X(1)+c.
	\end{equation}
	If this holds with $c = 0$ for all values of $h>0$ we say that $\mathbf{X}$ is strictly stable. For $\alpha=2$ the process $\mathbf{X}$ is $\alpha$-stable if and only if it is a Brownian motion with characteristic triplet $(\gamma,\Sigma,0)$. For $\alpha\in(0,2)$,  $\mathbf{X}$ is $\alpha$-stable if and only if its triplet is of the form $(\gamma,0,\Lambda)$ and the L\'evy measure is $-\alpha$-homogeneous, that is, $\Lambda(hE)=h^{-\alpha}\Lambda(E)$ for any $h>0$ and $E\in\B(\R^d)$. In this case, $\Lambda(\R^d) = \infty$. 
\end{example}

\begin{figure}[tb]
\begin{tikzpicture}
\matrix[column sep=.1cm]{
\begin{axis}[name=plot2,scale only axis,width=4.3cm,height=3cm,ymajorticks=false,xmin=0,xmax=1,ymin=-0.8,ymax=0.8,title=Brownian motion,ytick pos=left,xtick pos=bottom, xtick={0,1},title style = {text depth=0.5ex}]
\addplot+ [mark=none, color=AUblue]    table[x="X1",y="X2",col sep=comma] {BM.csv};
\end{axis}
&
\begin{axis}[name=plot1,scale only axis,width=4.3cm,height=3cm,ytick={0},xmin=0,xmax=1,xtick={0,1},ymin=-0.8,ymax=0.8,title=Compound Poisson process,ytick pos=left,xtick pos=bottom,title style = {text depth=0.5ex}]
	\addplot+ [jump mark left,mark options={fill=AUblue, scale=.5, shape=circle}, color=AUblue, thick]    table[x="X1",y="X2",col sep=comma] {CPP.csv};
	\draw [color=AUblue,thick] (axis cs: 0.978075725259259,-0.70794207580332) to (axis cs: 1,-0.70794207580332);
	\end{axis}
	&	
\begin{axis}[name=plot3,scale only axis,width=4.3cm,height=3cm,ymajorticks=false,xmin=0,xmax=1,ymin=-0.8,ymax=0.8,title=1-stable process,ytick pos=left,xtick pos=bottom, xtick={0,1},title style = {text depth=0.5ex}]
\addplot+ [color=AUblue, only marks,mark options={scale=.1}]    table[x="X1",y="X2",col sep=comma] {Cauchy.csv};
\end{axis}
\\
};
\end{tikzpicture}
\caption{Samples paths of three L\'evy processes corresponding to the Examples~\ref{ex:BM}-\ref{ex:stable} in dimension $d=1$.}
\label{fig:Levy_example}
\end{figure}
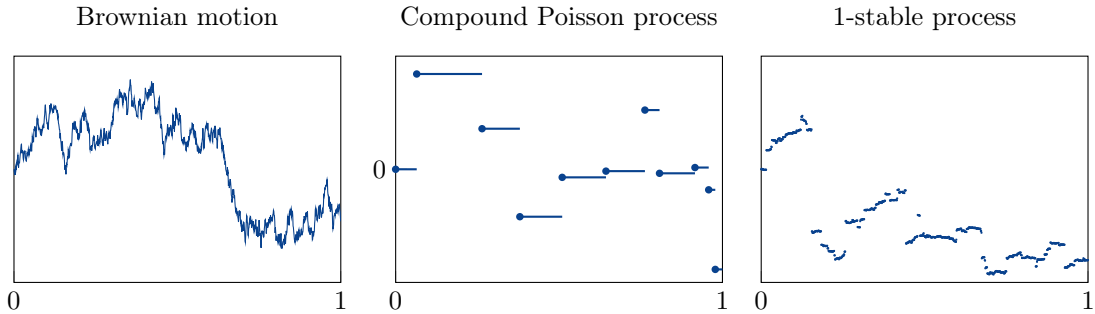

\subsection{Standardized L\'evy measure}
\label{sec:copula}

L\'evy processes $\mathbf X$ with L\'evy measure $\Lambda$ have been widely applied to model data in finance \cite{tankov2003financial} and many other fields.
While the continuous part $\mathbf W$ is fully characterized by $\gamma$ and $\Sigma$, the jump part $\mathbf J$ is more complex and cannot be described by a finite-dimensional parameter space. 
For statistical modeling it is therefore useful to have flexible parametric families of L\'evy measures. To this end, it is desirable to entangle the marginal distributions from the dependence structure.
The notion of copulas is applied for classical multivariate analysis for this purpose \citep[e.g.,][]{joe2014dependence}. For multivariate L\'evy processes, however, the copula $C_t$ of the random vector $X(t)$ is in general not time-invariant for different $t>0$ \citep{tankov_phd}.
Therefore, 
\cite{kallsen_tankov} introduce the L\'evy copula, which separates  the dependence structure of $\mathbf X$ from the distribution of the marginal processes, and characterizes it in terms of the L\'evy measure~$\Lambda$.

The L\'evy measure $\Lambda$ can have mass on the different orthants of $\mathbb R^d$, defined for any $s \in \{+,-\}^d$ as
\begin{align}\label{def_orthant}
	\mathcal O_s = \{x \in \mathbb R^d: s_i x_i > 0 \text{ or } x_i=0 \text{ if } s_i=+\}.
\end{align}
Note that with this definition, the orthants are disjoint, since each sub-face is contained in exactly one of them. We define the marginal tail functions $U_i:\mathbb R \to [0,\infty]$ for $i\in V$ as
\begin{align}
	\label{U_def} U_i(x) = \begin{cases}
		\Lambda_i(y \geq x), & x \geq 0,\\
		\Lambda_i(y \leq -x), & x < 0.
	\end{cases}
\end{align}
It follows from the theory of L\'evy copulas in \citet[Theorem 3.6]{kallsen_tankov} that there exists a L\'evy measure $\Lambda^*$ with
\begin{align}\label{levy_std1}
	\Lambda(y \in\mathbb R^d: s_i y_i \geq s_i x_i \text{ for } i\in V) = \Lambda^*(y \in\mathbb R^d: s_i y_i \geq s_i/U_i(x_i) \text{ for } i\in V), \quad x\in \mathcal O_s,
\end{align}
which satisfies the marginal constraints $\Lambda^*(y: y_i > x) = \Lambda^*(y: y_i < -x) = 1/x$ for $x > 0$ and $i\in V$.
We call $\Lambda^*$ a standardized L\'evy measure associated with $\Lambda$; see also \cite{eder_klup} for a similar concept on the positive orthant. The corresponding standardized L\'evy process $\mathbf X^*$ has Cauchy marginal processes and the same dependence structure (L\'evy copula) as the original process $\mathbf X$. Conversely, any standardized L\'evy measure $\Lambda^*$ together with any set of marginal tail functions $U_i$ of real-valued L\'evy processes, $i\in V$, define a L\'evy measure $\Lambda$ through~\eqref{levy_std1}.
The link with the L\'evy copula $C$ is explicitly given by~\citet[][Equation (3.2)]{kallsen_tankov} as
\begin{align}
	\Lambda^*(y \in\mathbb R^d: s_i y_i \geq s_ix_i \text{ for } i\in V) = \prod_{i\in V} \text{sign}(x_i) C(1/x_i, \dots, 1/x_d), \quad x\in\mathcal O_s.
\end{align}
We prefer to work with $\Lambda^*$, since we believe the interpretation is easier and it avoids notational difficulties.

\begin{remark}
	The standardized L\'evy measure, or, equivalently, the L\'evy copula, may not be unique for two reasons: if the marginal tail functions $U_i$ are not continuous for some $i\in V$; or if the marginal L\'evy measures do not explode at the origin, that is, either $\Lambda(x_i>0)<\infty$ or $\Lambda(x_i<0)<\infty$. In both cases, the range of the function $U_i$ on $x<0$ or $x>0$ is not equal to $(0,\infty)$ and therefore $\Lambda^*$ is not uniquely defined on certain parts of the domain. The first reason for non-uniqueness is similar to classical copulas. The second one only appears for L\'evy copulas. In most models used in practice, the marginal tail functions will be continuous and exploding, in which case $\Lambda^*$ and $C$ are unique.	
\end{remark}

In view of this remark, we will usually assume that each marginal process has infinite activity, that is, $\Lambda(x_i \neq 0) =\infty$ for all $i\in V$. In this case, \citet[][Theorem 27.4]{sato99} states that $F_{ti}$, the marginal distribution function of $X_i(t)$, is continuous for every $t>0$. Consequently, for any $t>0$, the ordinary copula $C_t$ of $X(t)$ is unique.
Interestingly, the standardized L\'evy measure $\Lambda^*$ can be obtained as the limit of the ordinary copulas of the L\'evy process $X(t)$ as $t\to 0$. Indeed, by \citet[Theorem 5.1]{kallsen_tankov}, we have 
\begin{align}\label{cop_approx}
	\Lambda^*(y: y_i \geq x_i \text{ for } i\in V) = \lim_{t\to 0} t^{-1} \mathbb P\left\{F_{ti}(X_i(t)) > 1 - t/x_i \text{ for } i\in V \right\}, \quad x > 0.
\end{align}
Similarly, we can obtain $\Lambda^*$ on the other orthants. 
This fact will be crucial for statistical inference since it provides a way to estimate the standardized L\'evy measure from observations of the process~$\mathbf X$. 

We give some examples for popular parametric classes of standardized L\'evy measures. To simplify notation, we concentrate here on models for the positive quadrant; general models can be defined similarly.  

\begin{example}\label{ex:log}
	The family of Clayton L\'evy measures with 1-stable margins is parameterized by $\theta \in (0,1)$. It is given by 
	\begin{align*}
		\Lambda^*(y: y_1 \geq x_1, \dots, y_d \geq x_d ) =  \left( x_1^{\theta}+\dots+ x_d^{\theta}\right)^{-1/\theta}, \quad x \geq 0;
	  \end{align*}
	  see \citet[][Example 6.2]{kallsen_tankov}.
\end{example}

One of the most popular models for exponent measures for multivariate extreme value distributions is the H\"usler--Reiss family \citep{HR1989}. This readily yields a corresponding family of L\'evy measures on the positive orthant. 
\begin{example}\label{ex:HR}
	The family of H\"usler--Reiss distributions is parametrized by a symmetric, strictly conditionally negative definite matrix $\Gamma$ in the cone
	\begin{align}
		\mathcal C^d
		=
		\left\{
			\Gamma \in \mathbb R^{d\times d}:		
			\Gamma = \Gamma^\top
			,\;\;
			\diag{\Gamma} = {0_d}
			,\;\;
			v^\top \Gamma v < 0
			\;\forall\,
			{0_d} \neq v \perp {1_d}
		\right\},
	\end{align}	
	where $0_d$ and $1_d$ are vectors in $\mathbb R^d$ with all zero and one entries, respectively.
	Equivalently, the model can be parametrized by the positive semi-definite precision matrix $\Theta = (P (-\Gamma/2) P)^+$, where $P = I_d - {1_d}{1_d}^\top /d$ is the projection matrix onto the orthogonal complement of ${1_d}$, and $A^+$ is the Moore--Penrose pseudoinverse of a matrix $A$; see \cite{hen_eng_seg} for details. We define the H\"usler--Reiss L\'evy measure $\Lambda^*$ through its density 
	\begin{align}\label{HR_density}
        \lambda^*(y)
        &\propto
	\prod_{i=1}^d y_i^{-1-1/d} \exp\left\{ -\frac12 (\log y - \mu_\Theta)^\top \Theta (\log y - \mu_\Theta) \right\}, \quad y \geq 0,
    \end{align}
	where $\mu_\Theta = P (-\Gamma/2) 1_d$.	
\end{example}

For more examples of parametric families of L\'evy measures we refer to \cite{kallsen_tankov}, \cite{bar2007}, \cite{klu2008} and  \cite{eder_klup}.

\section{Conditional independence for L\'evy processes}\label{levy_CI}

For random variables $X,Y,Z$, defined on some probability space $(\Omega,\ff,\p)$ and taking values in Polish spaces~$S_X,S_Y,S_Z$, we say that $X$ is conditionally independent of $Y$ given $Z$ if
\begin{equation*}
\Prob{X\in E_X, Y\in E_Y\given Z}=\Prob{X\in E_X\given Z}\Prob{Y\in E_Y\given Z}
\end{equation*}
almost surely for all measurable sets $E_X$ and $E_Y$. If this holds we write $X\Perp Y\mid Z$. This definition is general enough to cover cases of random vectors with values in $\mathbb R^d$ or stochastic processes with c\`adl\`ag sample paths in $\mathbb D$, for instance, since both are Polish spaces. For more details and properties of conditional independence we refer to \citet[Chapter~8]{kallenberg3}.

For disjoint subsets $A,B,C\subseteq V = \{1,\dots, d\}$ we want to study the conditional independence of a $d$-dimensional L\'evy process $\mathbf X$ between its components indexed by $A$ and $B$ given the components in $C$. Since random objects here are random processes, there are at least two ways to think about this conditional independence.
The first, and as we will argue, more natural perspective, requires
\begin{equation}\label{CI_process}
	\mathbf{X}_A\Perp\mathbf{X}_B\mid\mathbf{X}_C.
\end{equation}
This statement is on the level of the sample paths of the L\'evy process where we condition on the $\sigma$-algebra generated by the process $\mathbf X_C$, that is, the whole information of the evolution of components in $C$.
Alternatively, one may consider the corresponding conditional independence statement at a fixed time $t_0 > 0$, namely
\begin{equation}\label{CI_fixed}
	X_A(t_0)\Perp X_B(t_0)\mid X_C(t_0).
\end{equation}
Here, the conditioning includes only information in the $\sigma$-algebra $\sigma(X_C(t_0))$ generated by the components of the L\'evy process indexed by $C$ at the fixed time $t_0$. The latter is much smaller than $\sigma(\mathbf{X}_C)$. Indeed, we note that not only does $X_C(t_0)$ tell us nothing about the future $(X_C(t_0+t)-X_C(t_0))_{t\geq0}$, but, in general, it also provides insufficient information about the fluctuations and jumps up to time $t_0$.
Therefore, the notion in~\eqref{CI_fixed} seems unsuitable for applications and, moreover, it is unclear how it can be easily characterized. This is also in line with the discussion in \cite{kallsen_tankov} who argue that usual copulas for L\'evy processes at fixed times are not natural objects to study.  

In this section, we first provide a better understanding of the relation of the two statements~\eqref{CI_process} and~\eqref{CI_fixed}.
We then establish a fundamental result stating that the conditional independence on the sample path level in~\eqref{CI_process} can be studied separately for the Brownian part $\mathbf W$ and the jump part $\mathbf J$ in the L\'evy--It\^o decomposition.

\subsection{Conditioning on fixed times versus sample paths}

A first natural question is whether the conditional independence on the level of the sample paths in~\eqref{CI_process} is equivalent to the same statement at a fixed time $t_0$ in~\eqref{CI_fixed}. 
We show that this is in fact not true, since the process conditional independence does in general not imply conditional independence at fixed times.

For the simple case in $d=3$ dimensions with $A=\{1\}$, $B = \{3\}$ and $C = \{2\}$, Figure~\ref{fig:backdoor} provides a better intuition for the difference between conditioning on $X_2(t_0)$ as opposed to the whole path $\mathbf X_2$.
The solid edges represent direct dependencies while the dashed edges are the ones closed by the conditioning, that is, edges where at least one of the connected nodes is deterministic. Only conditioning on $X_2(t_0)$ then leaves open a "backdoor" between $X_1(t_0)$ and $X_3(t_0)$: since both random variables are dependent on their past processes $(X_1(t): t<t_0)$ and $(X_3(t): t<t_0)$, respectively, and since the latter two processes are dependent, fixing $X_2(t_0)$ does not suffice to block all of the dependence. On the other hand, conditioning on the entire path blocks the information flow and thus breaks the dependence between $X_1(t_0)$ and $X_3(t_0)$.

To make this more concrete, we provide an example where $X_A(t_0)$ and $X_B(t_0)$ have some dependence that is determined by the full history of $\mathbf{X}_C$ on $[0,t_0]$ and not only by $X_C(t_0)$.

\begin{example}\label{ex:fixed_time}
Let $\mathbf{X}$ be a 3-dimensional compound Poisson process such that the jumps of $\mathbf{X}_2$ take the values $-1$ or $1$ (both with positive probability).
Moreover, $\mathbf{X}_1$ and $\mathbf{X}_3$ are independent, jumping precisely when $\Delta\mathbf{X}_2=1$ and $\Delta\mathbf{X}_2=-1$, respectively. The jumps of $\mathbf{X}_1$ and $\mathbf{X}_3$ are all of size $1$.
Then $\mathbf{X}_1\Perp\mathbf{X}_3\mid\mathbf{X}_2$ because $\mathbf X_1$ and $\mathbf X_3$ are deterministic given $\mathbf X_2$. For any $t_0>0$ the random variables $X_1(t_0)$ and $X_3(t_0)$ are not conditionally independent given $X_2(t_0)$. For example, conditionally on the event $\Set{X_2(t_0)=0}$, which has positive probability, we have $X_1(t_0)=X_3(t_0)$ and this common value is not deterministic.
\end{example}

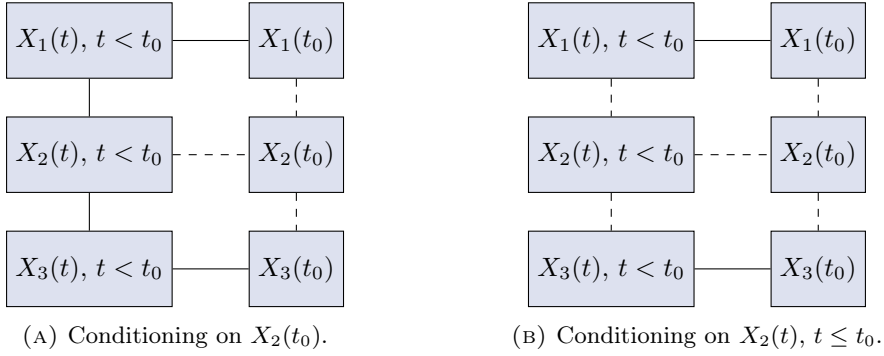
\begin{figure}[tb]
	\centering
	\begin{subfigure}[t]{0.45\textwidth}
	\centering
	\begin{tikzpicture}
	\node[fill=nodecolor,rectangle,draw,minimum size=1cm]   (1)                  {$X_1(t),\,t<t_0$};
	\node[fill=nodecolor,rectangle,draw,minimum size=1cm]	 (2)[right=of 1,]     {$X_1(t_0)$};
	\node[fill=nodecolor,rectangle,draw,minimum size=1cm]   (3)[below=.5cm of 1,]     {$X_2(t),\,t<t_0$};
	\node[fill=nodecolor,rectangle,draw,minimum size=1cm]   (4)[below=.5cm of 2,]     {$X_2(t_0)$};
	\node[fill=nodecolor,rectangle,draw,minimum size=1cm]   (5)[below=.5cm of 3,]     {$X_3(t),\,t<t_0$};
	\node[fill=nodecolor,rectangle,draw,minimum size=1cm]   (6)[below=.5cm of 4,]     {$X_3(t_0)$};

	\path
	(1) edge		 (2)
	(3) edge[dashed] (4)
	(5) edge		 (6)
	(1) edge		 (3)
	(3) edge		 (5)
	(2) edge[dashed] (4)
	(4) edge[dashed] (6)

	;
	\end{tikzpicture}
	\caption{Conditioning on $X_2(t_0)$.}
	\end{subfigure}
	\begin{subfigure}[t]{0.45\textwidth}
	\centering
	\begin{tikzpicture}
	\node[fill=nodecolor,rectangle,draw,minimum size=1cm]   (1)                  {$X_1(t),\,t<t_0$};
	\node[fill=nodecolor,rectangle,draw,minimum size=1cm]	 (2)[right=of 1,]     {$X_1(t_0)$};
	\node[fill=nodecolor,rectangle,draw,minimum size=1cm]   (3)[below=.5cm of 1,]     {$X_2(t),\,t<t_0$};
	\node[fill=nodecolor,rectangle,draw,minimum size=1cm]   (4)[below=.5cm of 2,]     {$X_2(t_0)$};
	\node[fill=nodecolor,rectangle,draw,minimum size=1cm]   (5)[below=.5cm of 3,]     {$X_3(t),\,t<t_0$};
	\node[fill=nodecolor,rectangle,draw,minimum size=1cm]   (6)[below=.5cm of 4,]     {$X_3(t_0)$};

	\path
	(1) edge		 (2)
	(3) edge[dashed] (4)
	(5) edge		 (6)
	(1) edge[dashed] (3)
	(3) edge[dashed] (5)
	(2) edge[dashed] (4)
	(4) edge[dashed] (6)
	;
	\end{tikzpicture}
	\caption{Conditioning on $X_2(t),\,t\leq t_0$.}
	\end{subfigure}
\caption{Illustration of the difference between conditioning on $X_2(t_0)$ and conditioning on $X_2(t),\,t\leq t_0$. To disconnect $X_1(t_0)$ and $X_3(t_0)$ it is not enough to condition on the present value $X_2(t_0)$.}
\label{fig:backdoor}
\end{figure}

Concerning the other direction, we can indeed show that conditional independence at all fixed times is strictly stronger than the corresponding statement on the sample paths level. 
\begin{proposition}\label{prop:fixed_times}
There is the implication
\begin{equation*}
X_A(t_0)\Perp X_B(t_0)\mid X_C(t_0)\text{ for all }t_0 > 0\qtq{$\Rightarrow$}\mathbf{X}_A\Perp\mathbf{X}_B\mid\mathbf{X}_C.
\end{equation*}
\end{proposition}
The proof of this proposition in Appendix~\ref{proof:fixed_times} does not require that $X_A(t_0)\Perp X_B(t_0)\mid X_C(t_0)$ for all $t_0 > 0$. In fact, it is sufficient for this to hold for all $t_0\in(0,\epsilon)$ for some $\epsilon>0$.

It is worthwhile to consider the case where $\mathbf{X}$ is a stable L\'evy process introduced in Example~\ref{ex:stable}. Because of~\eqref{eq:stable}, the left-hand side in Proposition~\ref{prop:fixed_times} is equivalent to having the conditional independence $X_A(t_0)\Perp X_B(t_0)\mid X_C(t_0)$ for a single time point $t_0>0$. Whether this is also true without stability is an open question. Importantly, distributional properties of $X(t)$ may be time dependent: \citet[\S23]{sato99} lists some examples such as unimodality and absolute continuity. 

A particular case when equivalence holds in Proposition~\ref{prop:fixed_times} is the Gaussian case. 
\begin{proposition}\label{Brownian_CI}
	Let $\mathbf X = \mathbf W$ have only a Brownian part, then    
	\begin{equation*}
		\mathbf{W}_A\Perp\mathbf{W}_B\mid\mathbf{W}_C \qtq{$\Leftrightarrow$} W_A(t_0)\Perp W_B(t_0)\mid W_C(t_0)\text{ for all }t_0 > 0.
		\end{equation*}
	Moreover, for the implication from right to left it suffices that there exists one $t_0 > 0$ where the conditional independence holds.
\end{proposition}

In applications, we might be interested in the L\'evy process $\mathbf X$ up to some finite time point. Conditional independence in~\eqref{CI_process} uses conditioning on the whole sample path, including the future of the process, which is typically not yet observed.  
The concept of killing is central in the theory of L\'evy processes as it allows us to consider $\mathbf{X}$ only observed until some time point $T>0$.
This time point is often taken to be a stopping time, but here we consider a deterministic, fixed time $T>0$ to define the {killed process} $\mathbf{X}^T$ by
\begin{equation*}
X^T(t)=\begin{cases}
X(t) & \text{for }t<T, \\
\dagger & \text{for }t\geq T,
\end{cases}
\end{equation*}
where $\dagger$ denotes a graveyard state. We then have the equivalence of the conditional independence on the level of the entire path and the killed paths. 
\begin{proposition}\label{prop:killed_process}
There is the equivalence:
\begin{equation*}
\mathbf{X}_A\Perp\mathbf{X}_B\mid\mathbf{X}_C\qtq{$\Leftrightarrow$}\mathbf{X}_A^T\Perp\mathbf{X}_B^T\mid\mathbf{X}_C^T\text{ for all }T>0.
\end{equation*}
\end{proposition}
Extending Proposition~\ref{prop:killed_process} to the case where $T$ can be any positive stopping time is not possible. The interesting implication is the one from left to right, and in general it does not hold. Indeed, suppose that  $\abs{A}=\abs{B}=1$ and $\mathbf{X}_C$ is independent of $\mathbf{X}_{A\cup B}$.
For $T=\inf\Set{t\geq1 : \abs{X_A(t)-X_B(t)}\leq\epsilon}$ with some $\epsilon>0$ we  readily see that $\mathbf{X}_A^T$ and $\mathbf{X}_B^T$ are not conditionally independent given $\mathbf{X}_C^T$. On the other hand, the implication from left to right in Proposition~\ref{prop:killed_process} does hold for all $\sigma(\mathbf{X}_C)$-measurable random variables $T$; in fact, the original proof applies without changes.

\subsection{The L\'evy--It\^o decomposition}

As mentioned in the preliminaries, thanks to the L\'evy--Itô decomposition the L\'evy process admits the representation $\mathbf{X}=\mathbf{W}+\mathbf{J}$ with Brownian part $\mathbf{W}$ and jump part $\mathbf{J}$.
A crucial question is how conditional independence for the sample paths of $\mathbf X$ relates to conditional independence of the sample path of the Brownian and jumps components. Surprisingly, it turns out that conditional independence can be studied separately for these two parts.
\begin{proposition}\label{prop:add_gaussian}
There is the equivalence
\begin{equation*}
\mathbf{X}_A\Perp\mathbf{X}_B\mid\mathbf{X}_C\qtq{$\Leftrightarrow$} \mathbf{W}_A\Perp\mathbf{W}_B\mid\mathbf{W}_C\qtq*{and}\mathbf{J}_A\Perp\mathbf{J}_B\mid\mathbf{J}_C.
\end{equation*}
\end{proposition}

A few comments are in order. While the Brownian part $\mathbf W$ and the jump part $\mathbf J$ in the L\'evy--Itô decomposition are always independent, this does not trivially imply the above result. In fact, for arbitrary independent random vectors $Y$ and $Z$ the implication
\begin{equation}\label{equiv_RV}
Y_A\Perp Y_B\mid Y_C\qtq*{and} Z_A\Perp Z_B\mid Z_C\qtq{$\Rightarrow$} Y_A+Z_A\Perp Y_B+Z_B\mid Y_C+Z_C
\end{equation}
does not hold in general. The issue is that conditioning on $Y_C+Z_C$ is not guaranteed to provide enough information about the individual terms $Y_C$ and $Z_C$. Also the other direction in~\eqref{equiv_RV} does not hold in general.
The proof of Proposition~\ref{prop:add_gaussian} in Appendix~\ref{proof:add_gaussian} makes use of the fact that $\mathbf{J}_C$ and $\mathbf{W}_C$ may be constructed from $\mathbf{X}_C$; see also Appendix~\ref{sec:decomp_cond} for details on this decomposition.

Conditional independence for the Brownian part is well understood. Suppose that the covariance matrix $\Sigma$ in the L\'evy triplet is invertible.
In light of Proposition~\ref{Brownian_CI} and classical results on Gaussian distributions \citep[e.g.,][]{lauritzen96}, we have for $i\neq j$ that
\[\mathbf{W}_i\Perp\mathbf{W}_j\mid\mathbf{W}_{V\setminus\{i,j\}} \qtq{$\Leftrightarrow$} (\Sigma^{-1})_{ij}=0, \]
and similarly for more general index sets.
Conditional independence for the jump part $\mathbf J$ has not been studied yet. The next section provides natural and easy-to-check conditions in terms of the L\'evy measure $\Lambda$.

\section{Conditional independence for the jump process}\label{sec:jump_part}

Now and throughout the rest of the paper, we consider a L\'evy process $\mathbf{X}$ with characteristic triplet $(\gamma,0,\Lambda)$, meaning that $\mathbf{X}$ has no Brownian component. 
If the L\'evy measure has finite total mass, $\Lambda(\mathbb R^d) < \infty$, or equivalently, if all marginal processes are compound Poisson, then studying conditional independence for the jump process $\mathbf X$ is fairly straightforward. In this case, one can consider $Y_A\Perp Y_B\mid Y_C$ where $Y$ is a random vector with jump distribution $\Lambda(\cdot)/\Lambda(\R^d)$.
In the more important case where $\Lambda$ has infinite total mass, this is obviously no longer possible. A different approach is required and we rely on the generalized notion of conditional independence for infinite measure introduced by \citet{eng_iva_kir}, which is perfectly suited for L\'evy measures. It is inspired by the work of \citet{engelke-hitz} who studied a new definition of conditional independence in the framework of multivariate extreme value theory.

In this section we first define $\Lambda$ conditional independence. A main contribution of the present paper is Theorem~\ref{thm:cond_indep_process_level} below, which links this notion to  conditional independence of the sample paths of $\mathbf{X}$. Moreover, we show that this conditional independence only depends on the standardized L\'evy measure but not the marginal tail functions, highlighting how it arises naturally as a dependence property. 

\subsection{Conditional independence for the L\'evy measure}

Recall the notation $V=\Set{1,\dotsc,d}$ and write
\begin{equation*}
\mathcal{R}(\Lambda)=\Set{R=\times_{v\in V}R_v\given R_v\in\B(\R),0\notin\bar{R},\Lambda(R)>0},
\end{equation*}
for the class of product sets that have positive mass and are bounded away from the origin; see Figure~\ref{fig:rectangles} for examples. Each set $R\in\mathcal{R}(\Lambda)$ has positive and finite mass, and therefore it induces a probability measure $\p_R=\Lambda(\cdot\cap R)/\Lambda(R)$. The reason for restricting to product sets is that it is natural in the context of independence. Indeed, two random vectors can only be independent if the joint support is a product set. We will often consider a random vector $Y$ with distribution $\p_R$.

\begin{figure}[tb]
\begin{tikzpicture}
\begin{axis}[name=plot1,axis lines=center, width=6cm,height=4.8cm, ymin=-2,ymax=6,xmin=-2,xmax=10,ticks=none, title=(A),unit vector ratio*=1 1 1]
\addplot+[name path=A,ForestGreen,domain=1:3,mark=none,solid] {1};
\addplot+[name path=B,ForestGreen,domain=1:3,mark=none,solid] {5};
\addplot[ForestGreen!50] fill between[of=A and B];
\draw[ForestGreen] (axis cs:1,1) to (axis cs:1,5);
\draw[ForestGreen] (axis cs:3,1) to (axis cs:3,5);

\addplot+[name path=C,ForestGreen,domain=4:8,mark=none,solid] {1};
\addplot+[name path=D,ForestGreen,domain=4:8,mark=none,solid] {5};
\addplot[ForestGreen!50] fill between[of=C and D];
\draw[ForestGreen] (axis cs:4,1) to (axis cs:4,5);
\draw[ForestGreen] (axis cs:8,1) to (axis cs:8,5);
\end{axis}

\begin{axis}[name=plot2,at={($ (.5cm,0cm) + (plot1.south east) $)},axis lines=center, width=6cm,height=4.8cm, ymin=-2,ymax=6,xmin=-2,xmax=10,ticks=none,title=(B),unit vector ratio*=1 1 1]
\addplot+[name path=A,white,domain=1:3,mark=none] {5};
\addplot+[name path=B,white,domain=1:3,mark=none] {3};
\addplot[BrickRed!50] fill between[of=A and B];
\addplot+[name path=C,white,domain=3:8,mark=none] {5};
\addplot+[name path=D,white,domain=3:8,mark=none] {1};
\addplot[BrickRed!50] fill between[of=C and D];
\draw[BrickRed!50] (axis cs:3,3) to (axis cs:3,5);
\addplot+[BrickRed,domain=1:8,mark=none,solid] {5};
\addplot+[BrickRed,domain=1:3,mark=none,solid] {3};
\addplot+[BrickRed,domain=3:8,mark=none,solid] {1};
\draw[BrickRed] (axis cs:1,3) to (axis cs:1,5);
\draw[BrickRed] (axis cs:3,1) to (axis cs:3,3);
\draw[BrickRed] (axis cs:8,1) to (axis cs:8,5);
\end{axis}

\begin{axis}[at={($ (.5cm,0cm) + (plot2.south east) $)},axis lines=center, width=6cm,height=4.8cm, ymin=-2,ymax=6,xmin=-2,xmax=10,ticks=none,title=(C),unit vector ratio*=1 1 1]
\addplot+[name path=A,white,domain=-1:5,mark=none] {-1};
\addplot+[name path=B,white,domain=-1:5,mark=none] {5};
\addplot[BrickRed!50] fill between[of=A and B];
\draw[black] (axis cs:0,-1) to (axis cs:0,5);
\draw[black] (axis cs:-1,0) to (axis cs:5,0);
\addplot+[name path=A,BrickRed,domain=-1:5,mark=none] {-1};
\addplot+[name path=B,BrickRed,domain=-1:5,mark=none] {5};
\draw[BrickRed] (axis cs:-1,-1) to (axis cs:-1,5);
\draw[BrickRed] (axis cs:5,-1) to (axis cs:5,5);
\end{axis}
\end{tikzpicture}
\caption{Examples of sets in $\R^2$. The sets in (A) and (B) are bounded away from~$0$ but only the former is a product set. (C) shows a product set containing~$0$. Hence, only the set in (A) is in $\mathcal{R}(\Lambda)$, provided that it has positive $\Lambda$ measure.}
\label{fig:rectangles}
\end{figure}
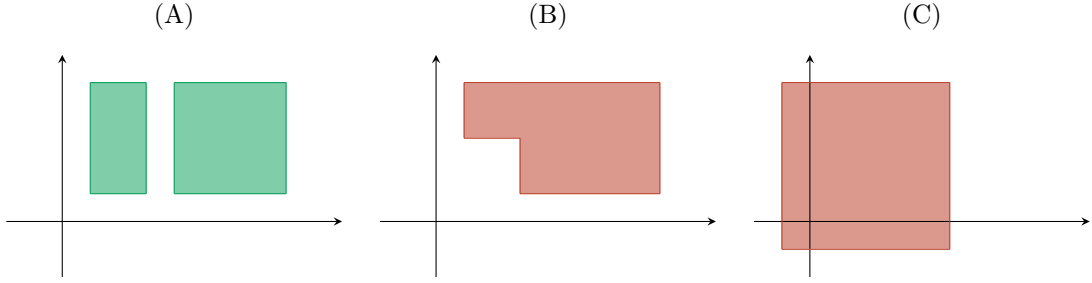

For a non-empty subset $D\subseteq V$ we may consider two measures on $\E^D=\R^D\setminus\Set{0_D}$ defined by
\begin{equation}\label{marginal_measures}
\Lambda_D=\Lambda(\Set{x_D\in\cdot}),\qquad \Lambda_D^0=\Lambda(\Set{x_D\in\cdot,x_{V\setminus D}=0_{V\setminus D}}).
\end{equation}
For $v\in V$ we shall use the simplified notation $\Lambda_v=\Lambda_{\Set{v}}$.
The distributions $\p_R$ are now used for defining conditional independence for the measure $\Lambda$.

\begin{definition}\label{def:cond_indep_measure}
  For disjoint sets $A, B, C\subseteq V$ that form a partition of $V$, we say that $\Lambda$ admits conditional independence of $A$ and $B$ given $C$, denoted by $\indepABC,$
  if we have the classical conditional independence
  \begin{align}\label{def_part}
  Y_A \Perp Y_B \mid Y_C \qquad \text{for } \quad Y\sim \mathbb P_{R} \quad \text{ for all } \quad R\in \mathcal R (\Lambda).
  \end{align}
  This is trivially true for $A$ or $B$ being empty, and for $C=\emptyset$ we say that $\Lambda$ admits independence of $A$ and $B$, and write $\indepAB.$
  If the sets $A, B$ and $C$ are not a partition of $V$, then the above definition remains the same with
  the test class in~\eqref{def_part} replaced by $ \mathcal R(\Lambda_{A \cup B\cup C})$.
\end{definition}

We recall some fundamental properties of this conditional independence notion. For statistical purposes, it is important to note that if $\Lambda$ possesses a Lebesgue density $\lambda$, then the conditional independence in~\eqref{def_part} for a partition $A,B,C\subseteq V$ is equivalent to the factorization
\begin{align}\label{dens_fact}
	\lambda(y) = \lambda_{A\cup C}(y_{A\cup C}) \lambda_{B\cup C}(y_{B\cup C}) / \lambda_C(y_C), \quad y\in \mathbb R^d \setminus \{0\},
\end{align}	
where $\lambda_I$ is the density of the marginal measure $\Lambda_I$ for some subset $I \subset V$; see \citet[][Theorem 4.6]{eng_iva_kir} for a slightly more general statement. Moreover, by \citet[][Lemma 3.3]{eng_iva_kir}, conditional independence of $\Lambda$ implies independence under the measure $\Lambda^0_{A\cup_B}$, that is,
\begin{equation}\label{supp_impl}
A\perp B\mid C\,[\Lambda]\qtq{$\Rightarrow$}A\perp B\,[\Lambda_{A\cup B}^0].
\end{equation}

\subsection{Distributional interpretation of the conditional independence}

For two subsets $A,B \subset V$, a well-known fact for a L\'evy process $\mathbf X$ with characteristic triplet $(\gamma, 0, \Lambda)$ is that the stochastic processes $\mathbf X_A$ and $\mathbf X_B$ are independent if and only if they have no simultaneous jumps \citep[][Proposition 5.3]{tankov2003financial}, that is,
\begin{equation}\label{levy_indep}
	\mathbf{X}_A\Perp\mathbf{X}_B\qtq{$\Leftrightarrow$}  \Lambda({x_A\neq0_A,x_B\neq0_B})=0.
\end{equation}
The conditional independence in Definition~\ref{def:cond_indep_measure} concerns the L\'evy measure $\Lambda$ and hence the jumps of the process $\mathbf X$. In this section we show how this translates to probabilistic statements about the process $\mathbf{X}$ itself, both for the example of independence above and more generally for conditional independence. We assume the explosiveness condition
\begin{equation}\label{eq:explosion_0}
\tag*{[A1]}
\Lambda({x_i\neq0})\in\Set{0,\infty}\For*{for all $i\in V$}.
\end{equation}
This assumption states that every one-dimensional L\'evy process $\mathbf X_i$ has either zero or infinite jump activity. This excludes the case of compound Poisson processes, for which a simpler conditional independence is sufficient. We begin by discussing the simpler case of independence. Under the explosiveness assumption above, \citet[Proposition~5.1]{eng_iva_kir} show that 
\begin{align}\label{Lambda_indep}
	A\perp B\,[\Lambda] \quad \Leftrightarrow \quad \Lambda({x_A\neq0_A,x_B\neq0_B})=0.
\end{align}
Together with~\eqref{levy_indep} we immediately have the following result.

\begin{lemma}\label{lem:independence_implication}
Let $A,B\subseteq V$ be a partition and assume \ref{eq:explosion_0}. Then there is the equivalence
\begin{equation*}
\mathbf{X}_A\Perp\mathbf{X}_B\qtq{$\Leftrightarrow$} A\perp B\,[\Lambda].
\end{equation*}
The implication from left to right holds even if \ref{eq:explosion_0} is false.
\end{lemma}

We now turn to the more difficult question of conditional independence.
For this, we need a slightly stronger assumption than \ref{eq:explosion_0}, namely
\begin{equation}\label{eq:explosion_1}
\tag*{[A2]}
\Lambda_D^0({x_i\neq0})\in\Set{0,\infty},\For*{for all $D\subseteq V$ and $i\in D$,}
\end{equation}
where $\Lambda_D^0$ is defined in~\eqref{marginal_measures}. The importance of this assumption was already noted in \citet[Theorem 5.3]{eng_iva_kir}, who show that it is sufficient for the $\Lambda$ conditional independence to form a semi-graphoid \citep[][Chapter 3]{lauritzen96}. This fact underlines that this conditional independence is natural and also implies various useful properties of the corresponding graphical models.

Lemma~\ref{lem:independence_implication} is a special case of the more general Theorem~\ref{thm:cond_indep_process_level} below, which states that $\mathbf{X}_A\Perp\mathbf{X}_B\mid\mathbf{X}_C$ if and only if the jumps of $\mathbf{X}$ satisfy the corresponding conditional independence statement.

\begin{theorem}\label{thm:cond_indep_process_level}
Let $A,B,C\subseteq V$ be a partition. Under \ref{eq:explosion_1} there is the equivalence
\begin{equation*}
\mathbf{X}_A\Perp\mathbf{X}_B\mid\mathbf{X}_C\qtq{$\Leftrightarrow$} A\perp B\mid C\,[\Lambda].
\end{equation*}
The implication from left to right holds without \ref{eq:explosion_1}.
\end{theorem}

The proof of this result is postponed to Appendix~\ref{proof:cond_indep_process_level}, but we sketch here the general idea, concentrating on the implication from right to left. It is possible to write $\mathbf{X}=\mathbf{X}'+\mathbf{X}''$, where $\mathbf{X}',\mathbf{X}''$ are independent L\'evy processes with L\'evy measures $\Lambda',\Lambda''$ given by $\Lambda$ restricted to $\Set{x_C=0}$ and $\Set{x_C\neq0}$, respectively. We then show that both processes satisfy the conditional independence between components $A$ and $B$ given $C$. For $\mathbf{X}''$ we know that all jumps have a non-zero $C$-component. Hence, conditioning on $\mathbf{X}_C''$ fixes all jump times. We combine this with the conditional independence $A\perp B\mid C\,[\Lambda]$. For $\mathbf{X}'$ we note that it is essentially a $(d-\abs{C})$-dimensional process since the $C$-component is deterministic. The associated L\'evy measure is $\Lambda_{V\setminus C}^0$ and due to \ref{eq:explosion_1} it satisfies \ref{eq:explosion_0}. Hence, we may apply Lemma~\ref{lem:independence_implication} since $A\perp B\mid C\,[\Lambda]$ implies $A\perp B\,[\Lambda_{V\setminus C}^0]$ by~\eqref{supp_impl}.

Conditional independence is a dependence property and, intuitively, should be independent of the marginal distributions.
For the classical case of $d$-dimensional random vectors with continuous marginals, indeed, two distributions that share the same copula also satisfy the same conditional independence statements. 
For a L\'evy process $\mathbf X$ with L\'evy measure $\Lambda$, the marginals are the univariate processes $\mathbf X_i$, $i\in V$, whose L\'evy measures are characterized by the tail functions $U_i$ defined in~\eqref{U_def}. The corresponding standardized L\'evy measure $\Lambda^*$ is defined in Section~\ref{sec:copula}.    
The next result shows that conditional independence statements for the sample paths of a L\'evy process with continuous marginal tail functions are the same for all processes sharing the same standardized L\'evy measure. This highlights again that the sample path conditional independence notion is natural for L\'evy processes.
\begin{proposition}\label{prop_copula}
Consider a standardized L\'evy measure $\Lambda^*$ and two sets of marginal tail functions $U_i$ and $U'_i$, $i\in V$, defining two L\'evy measures $\Lambda$ and $\Lambda'$, respectively. Assume that all $U_i(x)$ and $U'_i(x)$ are continuous and explode as $x\to \pm 0$. Then \[A \indep B \mid C\; [\Lambda]\qquad \Leftrightarrow\qquad A \indep B \mid C\; [\Lambda'].\]
\end{proposition}
It should be noted that both continuity and explosiveness of $U_i$ are essential. By allowing $U'_i$ to be non-explosive, only the implication from left to right would remain true in general.

\section{L\'evy graphical models}\label{sec:levy_GM}

Building on the understanding of conditional independence on the level of the sample paths of a L\'evy process, in Section~\ref{levy_graph} we define both undirected and directed graphical models for L\'evy processes. 
The equivalence of Theorem~\ref{thm:cond_indep_process_level} allows us to construct examples of such L\'evy graphical models on the level of the stochastic processes, or on the L\'evy measure $\Lambda$. The latter case is particularly useful when the models possess densities, since then conditional independence is readily checked through density factorization as in~\eqref{dens_fact}. Section~\ref{sec:stable} introduces an extension of $\alpha$-stable models and proposes flexible models on tree structures.

\subsection{Definition}\label{levy_graph}

Let $G=(V,E)$ be an undirected graph with nodes $V=\Set{1,\dotsc,d}$ and edge set $E \subset V\times V$, where a pair $(i,j)\in E$ represents an edge between the nodes $i$ and $j$. For undirected graphs the ordering does not matter, that is, $(i,j)$ and $(j,i)$ correspond to the same edge.
For sets $A,B,C\subseteq V$ we say that $C$ separates $A$ and $B$ in $G$ if every path from any node in~$A$ to any node in $B$ goes through $C$. In this case we write $A \indep_G B \mid C$.

\begin{definition}
	Let $\mathbf X$ be a L\'evy process with characteristic triplet $(0,0,\Lambda)$ whose L\'evy measure satisfies~\ref{eq:explosion_1}, and let $G=(V,E)$ be an undirected graph. Then $\mathbf X$ is an undirected L\'evy graphical model on $G$ if the global Markov property holds for $\Lambda$, that is, for any disjoint sets $A,B,C\subseteq V$ we have
	\[A \indep_G B \mid C \quad \Rightarrow \quad A \indep B \mid C\; [\Lambda].\]
	The right-hand side is equivalent to the conditional independence $\mathbf{X}_A\Perp \mathbf{X}_B\mid \mathbf{X}_C$.
\end{definition}

We now turn to the case of a directed graph $G = (V,E)$, where an element $(i,j)\in E$ represents a directed edge from node $i$ to $j$. We say that $G$ is a directed acyclic graph (DAG) if it does not contain any directed cycles. 
There are several directed Markov properties, which require the  notions of parents $\text{pa}(v)$, ancestors $\text{an}(v)$ and descendants $\text{de}(v)$ of a node $v\in V$ in graph $G$; see \cite{lauritzen96} for details.
We say that $\Lambda$ satisfies the directed local Markov property with respect to the DAG $G$ if for every node $v \in V$
\begin{align}\label{local_MP}
	 v \indep V\setminus \{ v \cup \text{de}(v) \cup \text{pa}(v) \} \mid \text{pa}(v) \; [\Lambda]. 
\end{align}
Under condition~\ref{eq:explosion_1}, \citet[Corollary 7.2]{eng_iva_kir} show that the directed local Markov property is equivalent to the directed global Markov property, a stronger condition that is closely linked to the notion of $d$-separation \citep{lauritzen96}. 

\begin{definition}
	Let $\mathbf X$ be a L\'evy process with characteristic triplet $(0,0,\Lambda)$ whose L\'evy measure satisfies~\ref{eq:explosion_1}, and let $G=(V,E)$ be a DAG. Then $\mathbf X$ is a directed L\'evy graphical model on $G$ if the directed local  Markov property~\eqref{local_MP} holds for $\Lambda$.
\end{definition}

We now give a generic construction principle for directed L\'evy graphical models on a DAG. 
\begin{example}
	Let $G = (V,E)$ be a DAG and for any $i\in V$, let $\mathbf Z_i = (Z_i(t))_{t\geq 0}$ be a univariate L\'evy process with triplet $(0,0,\Lambda_i)$ whose L\'evy measures satisfy~\ref{eq:explosion_0}. For coefficients $\beta_{ij} \in\mathbb R$, define the  L\'evy processes by the linear structural equation model
	\begin{align}
		\mathbf X_i = \sum_{j\in \text{pa}(i)} \beta_{ij} \mathbf X_j + \mathbf Z_i, \quad j\in V.
	\end{align} 
	 The resulting $d$-dimensional L\'evy process $\mathbf X = (\mathbf X_i : i\in V)$ has L\'evy triplet $(0,0,\Lambda)$ whose L\'evy measure $\Lambda$ satisfies~\ref{eq:explosion_1}. Denoting the matrix of coefficients by $B = (\beta_{ij})$ with $\beta_{ij}=0$ if $j\notin\text{pa}(j; G)$, we can write more compactly $\mathbf X = B \mathbf X + \mathbf Z$, which implies 
	\[\mathbf X = (I - B)^{-1} \mathbf Z;\]	
	see \cite{drt2017} for a similar argument in the case of classical Gaussian graphical models. Therefore, the L\'evy measure of the process $\mathbf X$ is supported on rays spanned by the columns of the matrix $(I - B)^{-1}$.
	Moreover, $\mathbf X$ is a directed L\'evy graphical model on the directed graph $G$, and an undirected L\'evy graphical model on the undirected graph $G^*$ obtained from moralizing $G$. This follows essentially from \citet[][Corollary 7.8]{eng_iva_kir}.
\end{example}

If all univariate processes $\mathbf Z_i$, $i\in V$, are $\alpha$-stable with the stability index $\alpha \in(0,2)$, then the $d$-dimensional random vector $X(1)$, that is, the L\'evy process $\mathbf X$ from the above construction at time 1, corresponds to the graphical model introduced in \cite{mis_kur_16}. Our example extends this to general L\'evy measures and to the level of stochastic processes. Note that for Gaussian case $\alpha = 2$, the dependence structure of a $d$-dimensional Brownian motion can always be represented with the above linear structural equation models. This is clearly not the case for general $\alpha$-stable processes, which can exhibit much richer dependence.

While the above example is a good illustration, it is not very useful for statistical inference. The reason is that the corresponding L\'evy measures $\Lambda$ are concentrated on rays and therefore have discrete spectral measure. Consequently, $\Lambda$ cannot have a Lebesgue density, which results in unrealistic dependence structures and inhibits likelihood inference. Our theory goes far beyond discrete spectral measures and we can indeed construct rich parametric classes of sparse L\'evy processes on graphical structures. This is studied in the next section.

\subsection{Stable tree models}\label{sec:stable}

The class of $\alpha$-stable L\'evy processes introduced in Example~\ref{ex:stable} have homogeneous L\'evy measures, which allow for stronger structural results of the corresponding graphical models. In this section we extend this model to heterogenous stable processes to make it suitable for applications. We then define a class of L\'evy graphical models on tree structures as a sparse yet flexible dependence model.

The stability index $\alpha\in(0,2]$ is essential in describing stable processes. For $\alpha=2$ the process $\mathbf{X}$ is $\alpha$-stable if and only if it is a Brownian motion, that is, its characteristic triplet is $(\gamma,\Sigma,0)$. For $\alpha\in(0,2)$, on the other hand, $\mathbf{X}$ is $\alpha$-stable if and only if its triplet is of the form $(\gamma,0,\Lambda)$ (i.e., no Brownian component) and the L\'evy measure is $-\alpha$-homogeneous; see Example~\ref{ex:stable}. For any $i\in V$ the marginal tail functions $U_i$ are therefore also $-\alpha$-homogeneous and consequently,
\begin{equation*}
U_i(\pm x)= c^\pm_i |x|^{-\alpha}, \quad x\geq 0, 
\end{equation*}
where $c^\pm_i= U_i(\pm 1) = \Lambda(\pm x_i > 1)$. 
We remark that \ref{eq:explosion_1} is automatically satisfied if $\Lambda$ is $-\alpha$-homogeneous.

A very useful property for statistical inference of an $\alpha$-stable process $\mathbf X$ is that the probabilistic copula $C_t$ of the random vector $X(t)$ is time-invariant, that is, $C_t = C_{t'}$ for all $t,t' > 0$. This directly follows from the stability property~\eqref{eq:stable}.
Below we provide a wider class of processes satisfying this property. These processes still have a stable copula, and stable marginals but with different stability indices, which is an important relaxation for practical applications. It should be noted that the property of $X(t)$ having a constant probabilistic copula across all $t>0$ is rather an exception for a general L\'evy process.

\begin{definition}\label{def_hetero}
	Let $\Lambda^*$ be a standardized L\'evy measure that is  $-1$-homogeneous, that is, $\Lambda(hE)=h^{-1}\Lambda(E)$ for any $h>0$ and $E\in\B(\R^d)$.
	We let the marginal stable tail functions be 
	\begin{align}\label{marginal_tail}
		U_i(\pm x)=c^\pm_i |x|^{-\alpha_i}, x \geq 0 \text{ for some }c^\pm_i > 0,\alpha_i\in(0,2).
	\end{align}
	This uniquely defines a L\'evy measure $\Lambda$ by~\eqref{levy_std1} on all orthants with marginal tail functions $U_i$ and dependence specified by $\Lambda^*$. We refer to a L\'evy process $\mathbf X$ with triplet $(0, 0, \Lambda)$ as a heterogenous stable process. 
\end{definition}

We can see the heterogenous stable L\'evy process as an extension of stable processes, but with the important difference for modeling in practice that the marginals can have different stability indices $\alpha_i\in(0,2)$. A crucial property of this model is that it inherits the copula time invariance of the stable process.

\begin{proposition}\label{prop_cop_same}
A heterogenous stable process $\mathbf X$ with L\'evy measure $\Lambda$  satisfies Assumption~\ref{eq:explosion_1} and has the stochastic representation 
\begin{align}\label{ss_hetero}
	X(t)\stackrel{d}{=}\left(t^{1/\alpha_i}X_i(1)+\gamma_i(t)\right)_{i\in V},
\end{align}
for some function $\gamma:[0,\infty) \to \mathbb R^d$. This implies that $\mathbf X$ has time-invariant copulas, that is, all $X(t)$ have the same probabilistic copula $C_t = C_1$, $t>0$.
\end{proposition}

Flexible parametric examples for heterogenous stable processes can be constructed easily. For instance, the standardized L\'evy measures in Examples~\ref{ex:log} and~\ref{ex:HR} are both $-1$-homogeneous and can therefore be combined with any marginal tail functions as in~\eqref{marginal_tail}. 
	
\begin{remark}\label{rem:asym}
	For applications, it is often infeasible to allow for different dependence structures on all $2^d$ orthants $\mathcal O_s$, $s\in\{+,-\}^d$. Instead, we can use the same measure for all orthants, but multiplied by different weights to introduce a certain asymmetry between orthants. In particular, in $d=2$ dimensions, let $\Lambda^*_+$ be a $-1$-homogeneous standardized L\'evy measure with mass only on the positive orthant $\mathcal O_{++}$, and define 
	\begin{align}
		\label{def2d}
		\Lambda^* (y: \pm y_1 \geq x_1, \pm y_2 \geq x_2) = m_{\pm, \pm} \Lambda^*_+(y: y_1 \geq x_1, y_2 \geq x_2), \quad x_1,x_2 \geq 0.
	\end{align}
	with weights $m_{++}, m_{--}, m_{+-}, m_{-+} \in [0,1]$ for the four orthants, respectively.
	Importantly, the marginal standardization of $\Lambda^*$ discussed in Section~\ref{sec:copula} requires that 
	\[m_{++} = m_{--} = 1- m_{+-} = 1-m_{-+},\]	
	that is, one parameter $m:= m_{++}\in [0,1]$ is sufficient to describe this model. In particular, if $m = 1/2$, there is symmetry between all orthants. If $m >1/2$, there is more mass on orthants $\mathcal O_{++}$ and $\mathcal O_{--}$, yielding positive dependence, and if $m < 1/2$ there is more mass on $\mathcal O_{+-}$ and $\mathcal O_{-+}$, corresponding to negative dependence.
\end{remark}

We discuss the case of heterogenous stable models on tree structures in more details, since our structure learning algorithm will concentrate on this class of graphs. A tree $T = (V,E)$ with nodes $V$ and edge set $E$ is a connected undirected graph without cycles. Trees are very sparse graphs since the number of edges $|E| = d-1$ is much smaller than then number of all possible edges $d(d-1)/2$.

Heterogenous stable models on tree have desirable properties for statistical inference. For instance, if the L\'evy measure $\Lambda$ admits a Lebesgue density $\lambda$ on $\mathbb R^d\setminus \{0\}$, then this density factorizes into lower-dimensional marginals as
\begin{align}\label{tree_fact} 
	\lambda(y) = \prod_{(i,j)\in E} \frac{\lambda_{ij}(y_i, y_j)}{\lambda_i(y_i)\lambda_j(y_j)} \prod_{i\in V} \lambda_i(y_i), \quad y \in\mathbb R^d\setminus\{0\},
\end{align}
where $\lambda_i$ and $\lambda_{ij}$, $i,j\in V$, are the univariate and bivariate marginal densities of $\Lambda$. The inverse of this result also holds: given a set of bivariate L\'evy measure densities with coinciding marginals on the intersection of the edges, \eqref{tree_fact} yields a valid $d$-dimensional L\'evy density $\lambda$ with corresponding L\'evy graph structure $T$ \citep[Theorem 6.4]{eng_iva_kir}; note that the integrability condition~\eqref{eq:integrability_cond} is satisfied because the marginal densities belong to valid L\'evy measures.
This density factorization is useful since it allows us to perform statistical inference separately for each bivariate model on the edge $(i,j)\in E$, which has huge computational advantages in higher dimensions.

We can also use~\eqref{tree_fact} to construct sparse L\'evy processes with tree structures. In particular, any bivariate parametric model (e.g., the Clayton model in Example~\ref{ex:log}) can be used for the $\lambda_{ij}$ and the conditional independence structure on $T$ yields a valid $d$-dimensional L\'evy process. A particularly important case arises if all bivariate distributions are from the H\"usler--Reiss model in Example~\ref{ex:HR}, since then the $d$-dimensional tree model is again in this model class.

\begin{example}
	Let $\Lambda^*$ be the standardized L\'evy measure of a H\"usler--Reiss model with paramter matrix $\Gamma$ as in Example~\ref{ex:HR} with mass on all orthants. Then $\Lambda^*$ is a L\'evy graphical model on the tree $T=(V,E)$ if and only if  $\Gamma$ satisfies the tree metric property
	\begin{align}\label{tree_metric}
		\Gamma_{hl} = \sum_{(i,j)\in \text{ph}(hl,T)} \Gamma_{ij},
	\end{align}
	where $\text{ph}(hl,T)$ is the unique undirected path from nodes $h$ to $l$ on tree $T$; see \cite{engelke-hitz} \cite{asenova2021inference} and \cite{engelke-volgushev} in the case of extremal tree models. Alternatively, we may also read off the conditional independence structure from the H\"usler--Reiss precision matrix, since $\Theta_{ij}=0$ whenever $(i,j)\notin E$ \citep{hen_eng_seg}.
\end{example}

In order to extend the asymmetry introduced in Remark~\ref{rem:asym} in two dimensions in a natural way to a general $d$-dimensional model, we rely on a L\'evy graphical model $\mathbf X$ on tree $T=(V,E)$ and the factorization~\eqref{tree_fact}.

\begin{definition}\label{def_hetero_tree}
	For each edge $(i,j)\in E$ of the tree $T=(V,E)$, let $\lambda^*_{ij}$ be the density of a model in~\eqref{def2d} with asymmetry parameter $m_{ij}\in[0,1]$ and some bivariate model on the positive orthant. Denote the corresponding standardized $d$-dimensional L\'evy measure implied by factorization~\eqref{tree_fact} by $\Lambda^*$. The latter is $-1$-homogeneous and its density on the orthant $\mathcal O_s$, $s\in\{+,-\}^d$, satisfies
	\begin{align}\label{tree_model}
		\lambda^*(x) = m_s \lambda_{(\text{sym})}^*(x), \quad x\in \mathcal O_s, \qquad   m_s = \prod_{(i,j)\in E}  m_{ij}^{\indnew\{s_i s_j > 0\}} (1-m_{ij})^{\indnew\{s_i s_j < 0\}}, 
	\end{align}
	where $\lambda^*_{(\text{sym})}$ is the symmetric standardized L\'evy density that arises from the choice $m_{ij} = 1/2$ for all $(i,j)\in E$.
	Together with an arbitrary choice of marginal tail functions as in~\eqref{marginal_tail}, we obtain a heterogenous L\'evy model $\Lambda$ that is a L\'evy graphical model on $T$.	
\end{definition}

The form of $m_s$ in~\eqref{tree_model} can be derived directly from the factorization~\eqref{tree_fact}. Importantly, in this model, the tree structure $T$ not only factorizes the density $\lambda^*$, but also yields a sparse representation of the asymmetry weighs $m_s$ on each orthant. This is a big simplification, since we only need $d-1$ parameters $m_{ij}$ on the edges $E$ of the tree to specify all $2^d$ weights on all orthants. This is particularly useful for statistical inference where the number of samples $n$ is typically much smaller $2^d$; see Section~\ref{sec:application}.

\section{Structure learning for trees}\label{sec:structure_learning_trees}

In this section, we study data-driven tree recovery for L\'evy tree models. 
Throughout, we assume that $\mathbf{X}$ is a $d$-dimensional L\'evy process with L\'evy measure $\Lambda$ following the heterogenous stable model in Definition~\ref{def_hetero_tree} with conditional independence structure given by the tree $T=(V,E)$. 
On a theoretical level, we first show that the tree structure is identifiable from bivariate summary statistics of the L\'evy copula. 
Based on this, we construct suitable estimators and propose a minimum spanning tree algorithm that consistently recovers the true underlying tree structure from discrete observations of the L\'evy process.

\subsection{Tree identifiability from bivariate L\'evy correlations}

We define the L\'evy correlation coefficient $\chi_{ij}$, $i,j\in V$, a bivariate statistic summarizing the dependence between two components $i$ and $j$ of the L\'evy process, by 
\begin{align}
	\label{chi_sym_def}
	\chi_{ij} := \Lambda^*(y: |y_i| > 1, |y_j| > 1)/2  \in [0,1], 
\end{align}
where $\Lambda^*$ is the standardized L\'evy measure in~\eqref{levy_std1} of $\Lambda$. Note that dividing by 2 together with the marginal standardization of $\Lambda^*$ discussed in Section~\ref{sec:copula} yields that $\chi_{ij}\in[0,1]$. 
This coefficient is an extension of the upper tail dependence coefficient introduced in \cite{eder_klup} for L\'evy measures with equal marginals. For our purpose, it is important to define it more generally for possibly heterogenous marginals by relying on the standardized L\'evy measure, and simultaneously for both large positive and negative jumps.
Larger values of $\chi_{ij}$ correspond to stronger dependence between the two components, and $\chi_{ij} = 0$ is equivalent to independence between the processes $\mathbf X_i$ and $\mathbf X_j$.

\begin{example}\label{ex:correlations}
	For parametric families it is typically easy to compute the L\'evy correlations. For instance, for the Clayton model in Example~\ref{ex:log} with symmetric mass on all orthants and parameter $\theta \in (0,1)$ we have $\chi_{ij} = 2 - 2^\theta$ for all $i,j\in V$. For the \HR{} model with symmetric mass on all orthants and parameter matrix $\Gamma$ as in Example~\ref{ex:HR}, the L\'evy correlation is $\chi_{ij} = 2 - 2\Phi(\sqrt{\Gamma_{ij}}/ 2)$, $i,j\in V$, where $\Phi$ is the standard normal distribution function.  
\end{example}

The next lemma will be crucial for building an estimator of this coefficient since it links the L\'evy correlation to a quantity that can be estimated from observations of the L\'evy process. 
\begin{lemma}\label{lem_chi}
	Let $\mathbf{X}$ be a L\'evy processes with L\'evy measure $\Lambda$ following the heterogenous stable model on the tree $T$ as in Definition~\ref{def_hetero_tree}. Then
	\begin{align}\label{chi_rep}
		 \chi_{ij} =  \lim_{q \to 1}  \chi_{ij}(q) = \lim_{q \to 1}  \frac{\mathbb P\{ |2F_{1i}(X_i(1)) - 1| > q, |2F_{1j}(X_j(1)) -1 | > q \}}{1-q}, 
	\end{align}
	where $F_{1k}$ is the distribution function of $X_k(1)$, $k\in V$.
\end{lemma}

We will now leverage the close connection of the L\'evy correlation coefficient to the extremal coefficient \citep{col1999, sch2003}, which is a measure for dependence in the distributional tails of a random vector. In particular, an important inequality for these coefficients in the context of extremal tree models was shown in \cite{engelke-volgushev}. Very similar arguments as therein yield inequalities for L\'evy correlations on tree models.
Here, $\phT(h,t)$ denotes the set of edges on the path from $h$ to $t$ on $T$. 
\begin{proposition}\label{prop:chi_ineq}
	Let $\mathbf{X}$ be a L\'evy process following the heterogenous stable model on the tree $T$ as in Definition~\ref{def_hetero_tree}. For any edge $(i,j)\in \phT(h,l)$, the L\'evy correlations satisfy 
	\begin{equation}
		\label{eq:chi_ineq}
		\chi_{hl}\leq \chi_{ij}. 
	\end{equation}
	Under the additional assumption that this inequality is strict as soon as $(i,j) \neq (h,l)$, the solution to the minimum spanning tree problem
	\begin{align}
		\label{mst}   
		T_{\MST} = \argmin_{T = (V,E)} \sum_{(i,j)\in E} -\log \chi_{ij},
	\end{align}
	is unique and satisfies $T_{\MST} = T$.
\end{proposition}

There are several sufficient conditions for inequality~\eqref{eq:chi_ineq} to be strict. For instance, it is sufficient that the support of the L\'evy measure is the whole space $\mathbb R^d\setminus \{0\}$ \citep{engelke-volgushev}, which is satisfied for our heterogenous stable model with Clayton or \HR{} distributions. An even weaker condition for strict inequality is obtained in \cite{hu_peng_segers}.

\subsection{Statistical inference for L\'evy tree models}

The previous section established that knowledge of the L\'evy correlations for all bivariate pairs of processes identifies an underlying tree structure uniquely. In practice, the L\'evy process $\mathbf{X}=(X(t))_{t\geq0}$ is not observed in continuous time but rather at a finite number of discrete time points. In this section we derive asymptotic guarantees on structure recovery for estimators $\widehat T$ based on empirical versions of the extremal correlations. In statistical theory for L\'evy processes there are two popular regimes on how the processes are observed. We briefly discuss both cases and show that they result in the same estimator.

The first, more classical regime is that the process is observed as a time series at regular time steps $t=0,1,\dots, n$, where each time point represents a unit such as an hour, a day or a year. Asymptotic theory is then derived under the assumption that $n\to \infty$.  For building an estimator, we consider the increments 
\begin{align}
	\label{incr_time_series}
	\Delta(t) = X(t) - X(t-1), \quad t=1,\dots , n.
\end{align}
From the defining properties of a L\'evy process in Section~\ref{sec:levy_def} we see that the increments $\Delta(1),  \dots, \Delta(n)$ are $n$ independent samples of the random vector $X(1)$, that is, the process at time one.

In order to obtain an empirical estimator $\widehat \chi_{ij}$ of the L\'evy correlation $\chi_{ij}$ between the processes $\mathbf X_i$ and $\mathbf X_j$, we rely on the limiting representation in Lemma~\ref{lem_chi}. In particular, for some $k \in \{1,\dots, n\}$, we let $q = 1- k/n$ be a probability level at which we estimate $\chi_{ij}(q)$. We replace the marginal distribution functions $F_i$ and $F_j$ of $X_i(1)$ and $X_j(1)$, respectively, by their empirical counterparts defined as  
\begin{equation*}
\widehat{F}_i(x)=\frac{1}{n+1}\sum_{s=1}^n \indnew{\Set{\Delta_i(s)\leq x}},\quad x\in\R,
\end{equation*}
and similarly for $\widehat F_j(x)$. The estimator of $\chi_{ij}$ is then
\begin{equation}\label{eq:chi_hat}
\widehat\chi_{ij}= \frac{1}{k}\sum_{t=1}^n\indnew{\Set{|2\widehat{F}_i(\Delta_i(t)) - 1| > 1 - k/n, |2\widehat{F}_j(\Delta_j(t))-1|>1 - k/n}}.
\end{equation}
where we omit the dependence on $n$ and $k$ for notational simplicity.
This estimator has the tuning parameter $k$, or equivalently the probability threshold $q=1-k/n$, which determines the number of samples used for estimation. The choice of $k$ is a bias-variance tradeoff since we can decompose the estimation error  of $\chi_{ij}$ almost surely as
\begin{align}\label{bias_var} | \widehat \chi_{ij} - \chi_{ij}| \leq | \widehat \chi_{ij} - \chi_{ij}(q)|  + | \chi_{ij}(q) - \chi_{ij}|. 
\end{align}
The first term on the right-hand side represents the variance of the estimator and can be controlled by requiring $k\to \infty$. The second term is the bias and requires that $q\to 1$, or equivalently, that $k/n \to 0$ as $n\to\infty$; see Lemma~\ref{lem_chi}. A very similar bias-variance tradeoff is common in extreme value statistics \citep{deh2006a}.

On the other hand, in financial applications, often a so-called high-frequency regime is considered, where the process $\mathbf X$ is sampled at fine grids with inter-observation times of seconds or even fractions of a second \citep[e.g.,][]{jac2012, bue2013}. To model such data mathematically, one might consider a compact interval with an increasingly finer grid of observations. At the $n$th step, the L\'evy process is assumed to be observed at times $t=0,1/n, \dots, (n-1)/n, 1$ and asymptotic theory assumes that $n\to\infty$.  
The increments in this case are defined as 
\begin{align}
	\label{incr_hf}
	\Delta(t,n) = X(t/n) - X((t-1)/n), \quad t=1,\dots , n.
\end{align}
While $\Delta(1,n), \dots, \Delta(n,n)$ are again independent, their distribution now equals the one of $X(1/n)$, the process at time $1/n$, and therefore it is changing in each step $n$ of the asymptotics.
Replacing the increments and empirical distribution function in~\eqref{eq:chi_hat} by the new increments~\eqref{incr_hf} and their empirical distribution function $\widehat F_{ni}$, now depending on $n$, yields a seemingly different high-frequency estimator $\widehat \chi_{ij}^{\text{HF}}$. Proposition~\ref{prop_hf} in the Appendix shows that the distribution of the estimators of the two asymptotic regimes coincide, that is, $\widehat \chi_{ij}^{\text{HF}} \stackrel{d}{=} \widehat \chi_{ij}$; this follows essentially from the extended stability property~\eqref{ss_hetero} for heterogenous stable process. It therefore suffices to study the asymptotic properties of $\widehat \chi_{ij}$.

The next theorem connects the estimator $\widehat \chi_{ij}$ with a well-known estimator for extremal correlation in extreme value theory and leverages existing results to show consistent tree recovery for the minimum spanning tree based on these empirical L\'evy correlations.

\begin{theorem}\label{thm:recovery}
	Let $\mathbf{X}$ be a L\'evy processes following the heterogenous stable model on the tree $T$ as in Definition~\ref{def_hetero_tree}. Suppose that the inequalities in~\eqref{eq:chi_ineq} are strict if $(i,j) \neq (h,l)$. Let the intermediate sequence satisfy $k\to\infty$ and $k/n\to 0$, as $n\to \infty$. Then, the estimator of the tree structure
	\begin{align}\label{mst_emp}
		\widehat T = \argmin_{T = (V,E)} \sum_{(i,j)\in E} -\log \widehat \chi_{ij},
	\end{align}
	based on the empirical L\'evy correlations~\eqref{eq:chi_hat} consistently recovers the true underlying tree, that is,
	\[ \lim_{n\to\infty} \mathbb P( \widehat T = T) = 1.\]
\end{theorem}

The proof of this theorem relies on splitting the L\'evy correlation into the contributions from the four orthants, that is, 
\begin{align}\label{chi_split_population}
	\chi_{ij} = \chi^{++}_{ij} + \chi^{+-}_{ij} + \chi^{-+}_{ij} + \chi^{--}_{ij}, 
\end{align}
where $\chi^{++}_{ij} = \Lambda^*(y: y_i > 1, y_j > 1)/2$ and similarly for the other three components. The empirical counterparts of these four terms can be defined similarly to~\eqref{eq:chi_hat}; for details see the proof of Theorem~\ref{thm:recovery}. The consistency of each term can then be shown relying on results from extreme value theory.

\begin{remark}
	The connection between L\'evy processes and extreme value theory arises naturally since an infinite measure $\Lambda$ exploding at the origin is central in both areas \citep{Kabluchko2009, wan2010}. Combinations of the two fields are studied for stationary particle systems \citep{engelke2015} and max-stable processes \cite{engelke2016}.
\end{remark}

Once the underlying tree structure $T=(V,E)$ is established, the asymmetry parameters $m_{ij}$, $(i,j)\in E$ of the model in Definition~\ref{def_hetero_tree} have to be estimated. 
From~\eqref{def2d} and~\eqref{chi_sym_def}, we see that a candidate for an estimator of $m_{ij}$ is
\begin{align}\label{m_est}
	\widehat m_{ij} = \frac{\widehat \chi_{ij}^{++}+ \widehat \chi_{ij}^{--}}{\widehat \chi_{ij}}.	
\end{align}
The next proposition shows that this is indeed a consistent estimator for the asymmetry parameters.

\begin{proposition}\label{prop:m_consistency}
	Under the same assumptions as in Theorem~\ref{thm:recovery}, the estimator~\eqref{m_est} of the asymmetry parameter $m_{ij}$, $(i,j)\in E$ is consistent, that is, for any $\varepsilon >0$ we have
	\[ \lim_{n\to\infty} \mathbb P( | \widehat m_{ij} - m_{ij} |>\varepsilon) = 0.\]
\end{proposition}

We note that in this section, we have not assumed any particular parametric models for the densities $\lambda^*_{ij}$ in the heterogenous stable model since our structure estimation is fully non-parametric. If we assume a parametric model such as in Example~\ref{ex:log} or \ref{ex:HR}, we can 
obtain an estimate of the model parameters by inverting the closed-form representations of $\chi_{ij}$ in Example~\ref{ex:correlations}, for instance.

\section{Experiments}\label{sec:experiments}

\subsection{Simulation study}\label{sec:simu}

We describe how to simulate approximate sample paths of a heterogenous L\'evy process $\mathbf X$ with characteristic triplet $(\gamma,0,\Lambda)$, where $\gamma\in\mathbb R^d$ is the drift and $\Lambda$ is the L\'evy measure of  heterogenous L\'evy process $\mathbf X$ as in Definition~\ref{def_hetero}. Recall that $\Lambda^*$ in~\eqref{levy_std1} denotes the standardized L\'evy measure, and $\alpha_i\in (0,2)$ and $c_i^\pm > 0$ are the marginal parameters in~\eqref{marginal_tail}.
Since it is in general not possible to simulate L\'evy processes with infinite jump activity exactly, we follow the usual approach of 
approximating $\mathbf X$ by an auxiliary L\'evy process $\mathbf{X}\ph{(\epsilon)}$ for some $\epsilon >0$. To this end, define the norm $\|x\|_{c,\alpha} = \| (cx)^{1/\alpha} \|_\infty$, where the notation $(cx)^{1/\alpha}$ has to be read componentwise such that, for instance, for $x_i < 0$ the $i$th component is $(c_i^- x_i)^{1/\alpha_i}$. 
The auxiliary process is then defined as 
\begin{equation}\label{approx_process}
	X\ph{(\epsilon)}(t)=\gamma\ph{(\epsilon)} t+Y\ph{(\epsilon)}(t), \quad t\geq 0, 
\end{equation}
where the drift and the compound Poisson process $\mathbf Y\ph{(\epsilon)}$ have representations
\begin{equation*}
\gamma\ph{(\epsilon)}=\gamma-\int_{\Set{\epsilon\leq \|x\|_{c,\alpha}\leq 1}}x\,\Lambda(\dd x), \qquad Y\ph{(\epsilon)}(t) = \sum_{i=1}^{N(t)} Z_i,
\end{equation*}
respectively.
Here $\mathbf N$ is a homogeneous Poisson process with rate $\Lambda({\norm{x}_{c,\alpha}\geq \epsilon}) =\Lambda^*({\norm{x}_{\infty}\geq \epsilon})  $ and the $Z_i$ are independent samples from the probability distribution 
\[\p_\epsilon(A)=\Lambda(A\cap\Set{\norm{x}_{c,\alpha}\geq\epsilon})/\Lambda(\Set{\norm{x}_{c,\alpha}\geq\epsilon}), \quad A \subset \mathcal B(\mathbb R^d).\] 
The crucial ingredient for both parts of the approximation are exact samples from the measure~$\p_\epsilon$, since it allows us to estimate $\gamma^{(\epsilon)}$ via Monte Carlo methods and to simulate paths of $\mathbf Y^{(\epsilon)}$. For a wide range of parametric models for the L\'evy measure $\Lambda$, including the Clayton and the H\"usler--Reiss distributions, we can adopt the simulation technique in~\cite{dom2016} based on so-called extremal functions. Similarly as in \citet[Lemma 2]{engelke-hitz}, we may use rejection sampling to obtain a realization $Z^*$ of the distribution
\[\p^*_\epsilon(A)=\Lambda^*(A\cap\Set{\norm{x}_\infty\geq\epsilon})/\Lambda^*(\Set{\norm{x}_\infty\geq\epsilon}), \quad A \subset \mathcal B(\mathbb R^d);\]
this is implemented in the \texttt{graphicalExtremes} R package \citep{graphicalExtremes2022}. 
We can then transform to the original margins by defining a random vector $Z$ with entries
\[Z_j = \indnew{\Set{Z_j^* \geq 0}} (c_j^+Z_j^*)^{1/\alpha_j} - \indnew{\Set{Z_j^* < 0}} (-c_j^-Z_j^*)^{1/\alpha_j}, \quad  j\in V.\]
One can check that $Z \sim \mathbb P_\epsilon$.
The approximation is suitable since it converges in distribution to the target L\'evy process, that is, $\mathbf{X}\ph{(\epsilon)}\convd*\mathbf{X}$ as $\epsilon \to 0$ \citep[Chapter 7]{kallenberg3}. 

\begin{remark}\label{rem:small_jumps}
	The approximation~\eqref{approx_process} can be refined by introducing
	a $d$-dimensional Brownian motion with a certain covariance matrix $\Sigma$ to approximate the small jumps of $\mathbf X$. This can improve the convergence rate to the target process \citep{cohen-rosinski}.
	It is natural to ask whether the Brownian motion of this small jump approximation inherits the conditional independence statements from $\Lambda$. In Appendix~\ref{small_jump_approx} we show that under certain rather strict assumptions, this is indeed the case; generally we believe that it does however not hold.
\end{remark}

We perform a simulation study by first generating a random tree structure $T=(V,E)$ in dimensions $d\in \{10,30\}$. We then generate the parameter matrix $\Gamma$ of a \HR{} distribution by sampling the parameters $\Gamma_{ij}$ on the edges $(i,j)\in E$ from a uniform distribution $\text{Unif}[1,6]$. The remaining entires of $\Gamma$ are implied by the tree metric property~\eqref{tree_metric}. We then generate a discrete sample of the L\'evy process $\mathbf X$ with characteristic triplet $(0,0,\Lambda)$ at time points $t=0,1,\dots, n$, where $\Lambda$ follows a $-1$-homogeneous symmetric \HR{} model with parameter $\Gamma$; note that this is a heterogenous stable model on $T$ with $\alpha_i=1$, $i\in V$, as described in Definition~\ref{def_hetero_tree}. For the simulation we use the approximative method above. 
We then estimate the L\'evy correlations $\widehat \chi_{ij}$, $i,j\in V$, as in~\eqref{eq:chi_hat} for different probability thresholds $q = 1-k/n$ and the corresponding minimum spanning tree $\widehat T$ as in~\eqref{mst}. 

We repeat the whole procedure $500$ times. Figure~\ref{fig:sim_study} shows the proportion of times that the true underlying tree is recovered by our learning method, that is, $\widehat T = T$. Results are shown for different sample sizes $n$ and the different probability thresholds $q$.
We observe a clear bias-variance tradeoff as described in~\eqref{bias_var}: when the threshold $q$ is too small, the estimates $\widehat \chi_{ij}$ are biased and the tree is estimated poorly; if $q$ is very close to 1, then the variance of the $\widehat \chi_{ij}$ is too large. A good compromise in this case is in the interval $[0.8, 0.95]$, but in general the optimal choice will depend on the dimension $d$ and the parameter matrix $\Gamma$. We also observe that, as expected, estimation of the tree structure is easier in lower dimensions, in the sense that it requires smaller sample sizes.

\begin{figure}[tb]
	\includegraphics[width=0.49\textwidth]{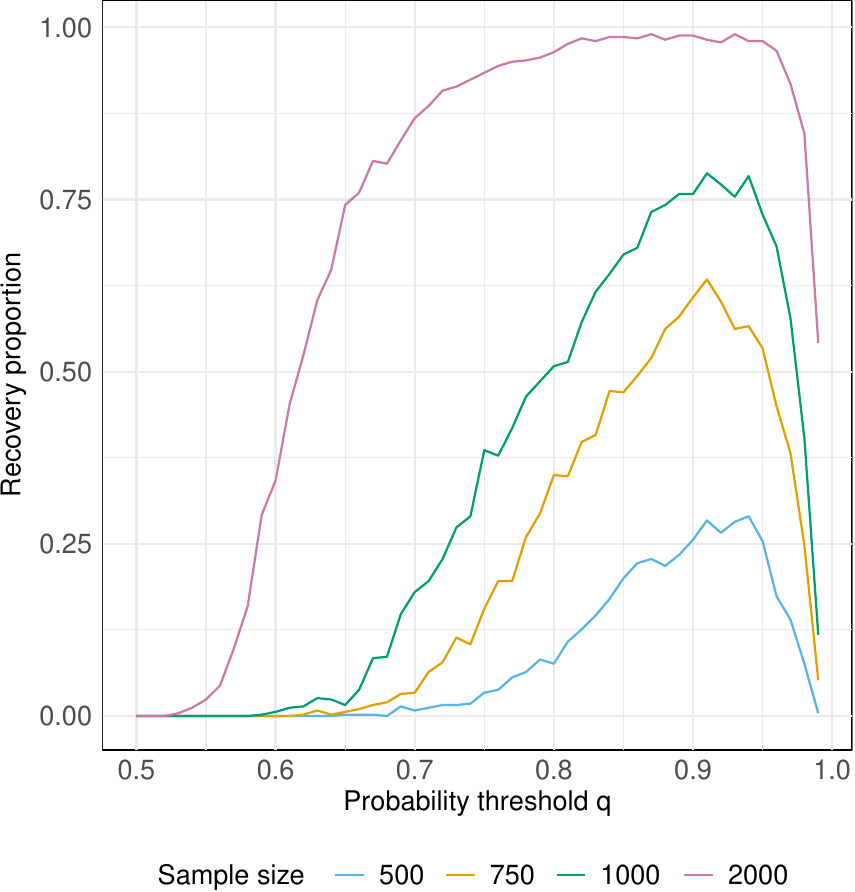}
	\includegraphics[width=0.49\textwidth]{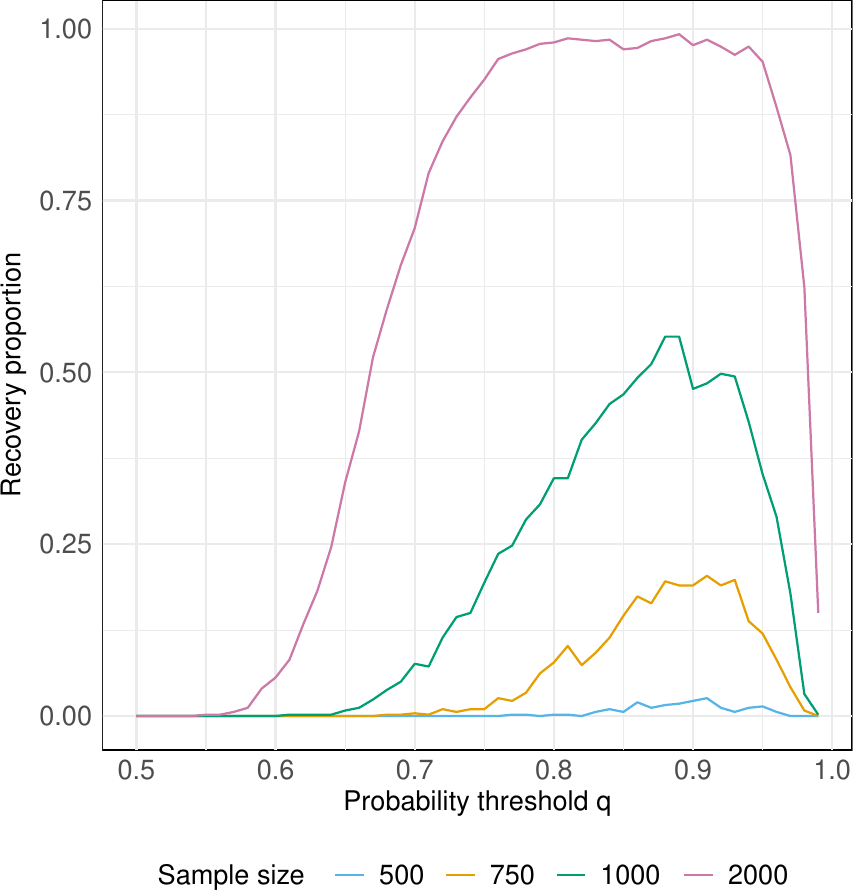}
	\caption{Proportion of correctly recovered tree structures $\widehat T = T$ as a function of the probability threshold $q = 1-k/n$ for data in $d=10$ (left) and $d=30$ (right) dimensions and differnt sample sizes $n$.}
	\label{fig:sim_study}
\end{figure}

\subsection{Application to financial data}\label{sec:application}

We apply our methodology for sparse dependence modeling of L\'evy processes and the data-driven structure learning for trees on 
daily stock prices for $d=16$ American companies during the period April 1, 2010, until December 31, 2015; see Table~\ref{table:companies} for a description of the stocks and their respective industry sectors. Following the Black--Scholes model for L\'evy processes \citep[e.g.,][Chapter 11]{tankov2003financial}, we assume that the stock prices $\mathbf S = (S(t): t\geq 0)$ can be modelled by the exponential of a $d$-dimensional L\'evy process $\mathbf X$, that is,
\[ S(t) = S(0)  \exp\{\gamma t + X(t) \}, \quad t\geq 0,\]
where $S(0)$ are the initial stock prices and $\gamma\in\mathbb R^d$ is a drift. We further assume that the L\'evy process $\mathbf X$ is a heterogenous stable model on some (unknown) tree $T = (V,E)$ as in Definition~\ref{def_hetero_tree}. In particular, the marginal processes $\mathbf X_i$ can have different stability indices $\alpha_i$ and asymmetry parameters $c_i^+$ and $c_i^-$, $i\in V$.

We have $n=1509$ daily observations of the stock prices. From this we compute the daily log-returns 
\[\log \{ S(t)/S(t-1)\}, \quad t=1,\dots, n,\]
which according to our model are independent observation of $X(1)$; recall the stationary and independent increments of the L\'evy process $\mathbf X$ in Section~\ref{sec:levy_def}. 
The right-hand side of Figure~\ref{fig:3D_path} shows the time series of the stock prizes of the three companies U.S.~Bancorp, Cisco Systems and Wells Fargo. To explore the dependence between the performance of the different companies, the top row of Figure~\ref{fig:scatter} shows scatter plots of the log-returns of all pairs of these companies. We can generally see that there seem to be returns, that is, increments of the L\'evy process, in all four orthants, indicating that also the standardized L\'evy measure $\Lambda^*$ has mass on all orthants; see approximation~\eqref{cop_approx} and Section~\ref{sec:copula}.
The second observation is that the dependence is highly asymmetric between the different orthants, with much larger mass on $\mathcal O_{++}$ and $\mathcal O_{--}$, indicating positive dependence. Our model in Definition~\ref{def_hetero_tree} therefore seems suitable since it has the flexibility to capture these features.

\begin{figure}[ht]
	\includegraphics[width=1\textwidth]{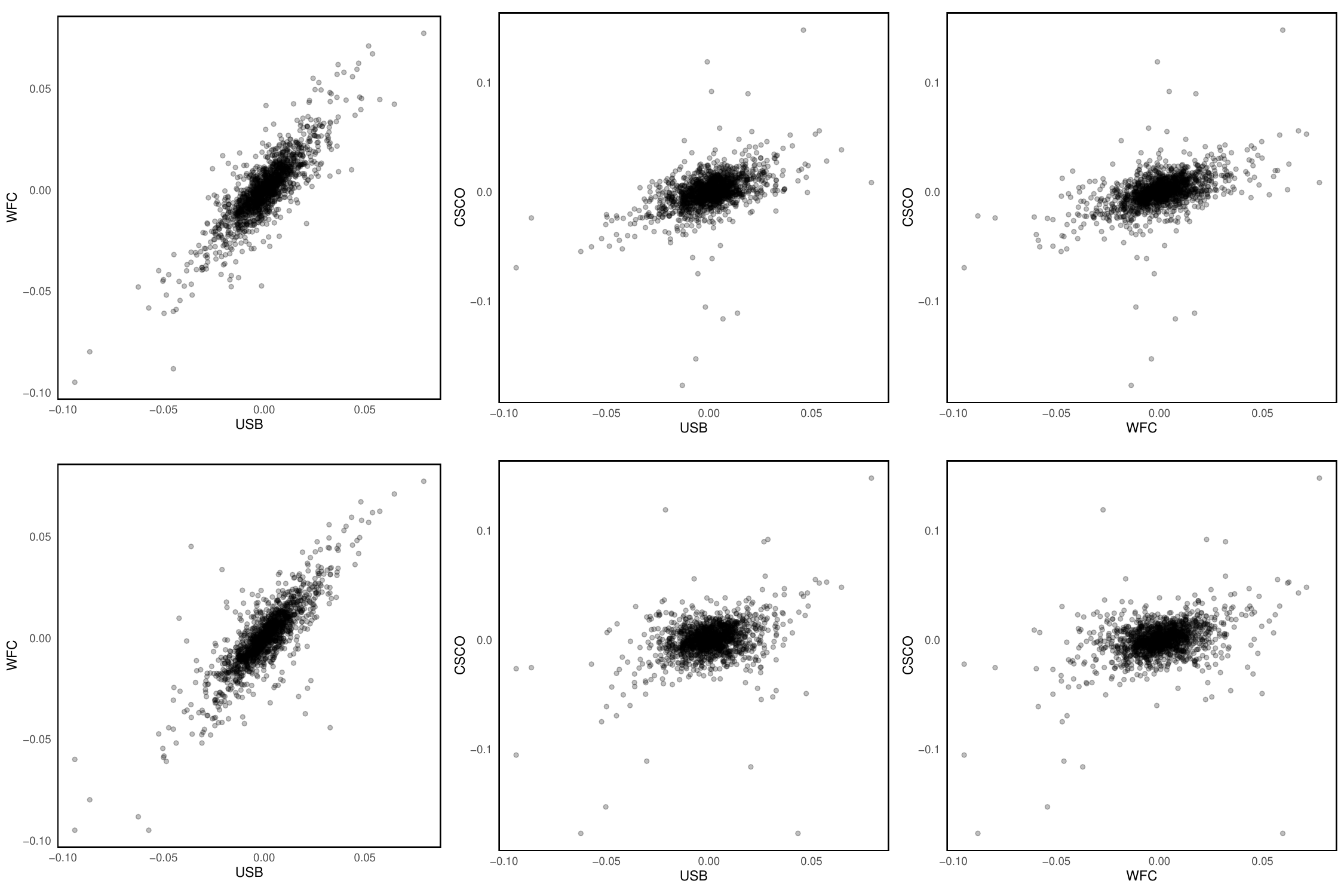}
	\caption{Scatter plots of {original daily stock returns} (top) and data simulated from {fitted L\'evy tree model} (bottom). Distance between nodes on $\widehat T$ is 1, 3 and 7 from left to right.}
	\label{fig:scatter}
\end{figure}

With these observations, we first apply our minimum spanning tree algorithm to learn the underlying tree structure $T$. Note that Theorem~\ref{thm:recovery} guarantees consistent recovery of the underlying tree without knowledge of the asymmetry parameters $m_{ij}$, $(i,j)\in E$. For this step it suffices to compute an estimate of the L\'evy correlation matrix $\widehat \chi_{ij}$, $i,j\in V$, based on the log-returns. The corresponding minimum spanning tree $\widehat T = (V, \widehat E)$ in~\eqref{mst_emp} is shown on the left-hand side of Figure~\ref{fig:est_tree}. The nodes in the graph are color-coded, with colors corresponding to the respective industry sectors of the companies.
We observe a nicely interpretable structure of the dependence between the stock returns. Indeed, without providing this information to the algorithm, in the estimated tree, stocks from the same industry tend to be closer (in terms of shortest paths on $\widehat T$) than stocks from different industries. This confirms the intuition that returns of companies in the same industry should be more directly dependent due to industry specific risks and drivers. 
In order to assess the uncertainty of the tree estimation, we subsample the original daily log-returns with subsample size $\lceil n/2\rceil $ and re-estimate the tree structure. We repeat this 300 times. The right-hand side of Figure~\ref{fig:est_tree} shows the tree where the width of each edge indicates the number of times this edge was contained in the estimated tree. Overall, the main structure $\widehat T$ seems to be fairly stable, especially the links within industries.

To assess whether this tree yields a reasonable model for the dependence dependence structure observed in the data in the top row of Figure~\ref{fig:scatter}, we have to specify a parametric model for the bivariate distributions on the edges of the tree $\widehat T$. We choose bivariate \HR{} distributions from Example~\ref{ex:HR} because its flexibility, but also since the corresponding $d$-dimensional model will again be of this class. Conveniently, consistent estimates of the \HR{} parameters on the edges of the tree can be directly obtained from the L\'evy correlations by
\[ \widehat \Gamma_{ij} = \left\{2\Phi^{-1}\left(1 - \widehat \chi_{ij}/2\right)\right\}^2, \quad (i,j) \in \widehat E;\]  
see Example~\ref{ex:correlations}. The tree structure $\widehat T$ then implies the parameters $\widehat \Gamma_{ij}$ also for non-edges $(i,j)\notin E$ by the tree metric property~\eqref{tree_metric}.

The only missing parameters for the dependence model in Definition~\ref{def_hetero_tree} are the asymmetry parameters $m_{ij}$ on the edges $(i,j)\in \widehat E$.	Following the consistency result in Proposition~\ref{prop:m_consistency}, we obtain estimates $\widehat m_{ij}$ using the empirical estimator in~\eqref{m_est}. Interestingly, all estimates $\widehat m_{ij}$, $(i,j)\in \widehat E$, are contained in the interval $[0.8, 1]$, which corresponds to strong positive dependence. This agrees with the initial observation from Figure~\ref{fig:scatter}.

Our fitted dependence model can be used to generate new sample paths of stock prices of the 16 selected companies. Such simulations can be used for stress testing and risk assessment. In addition, they help to assess the model fit. Up to now, we have operated on the standardized scale of the standardized L\'evy measure $\Lambda^*$. In order to obtain samples of the stock processes $\mathbf S$ on the original scale, we need to transform back to the original marginal distributions. One option would be to estimate the marginal parameters $\alpha_i\in (0,2)$ and $c^+_i, c^-_i > 0$, $i\in V$, of the heterogenous stable model in Definition~\ref{def_hetero}; see \cite{zol1986} for classical methods. Since we are interested in dependence modeling, we choose to use the empirical versions of the marginal tail functions $U_i$ based on the observed increments in each components. More precisely, we generate a path of $\mathbf X^*$ under the standardized L\'evy measure $\Lambda^*$, and then map the increments of the $i$th marginal process to those of the observed log-returns of the $i$th stock with the same rank, that is, the largest to the largest, etc. This couples the marginal processes to the observations but uses the dependence structure from $\Lambda^*$; see \cite{FOMICHOV2021407} for a related coupling in one-dimensional L\'evy processes.
This approach is also similar to using empirical distribution functions in classical multivariate dependence modeling. 
The bottom row of Figure~\ref{fig:scatter} shows the scatter plot of simulated log-returns of three pairs of stocks. We can see that the model is able to capture the strength of dependence as well as the asymmetry very well. This is not only true for stocks that are adjacent in the tree $\widehat T$ and whose dependence is directly modelled (e.g., U.S.~Bancorp and Wells Fargo), but also for stocks that are far away in $\widehat T$ (e.g., Wells Fargo and Cisco Systems) and whose dependence is therefore implied by the underlying conditional independence structure. This shows that the tree assumption yields a reasonable model for this data set. Moreover, the corresponding sample paths of our fitted stock price process $\mathbf S$ is shown on the left-hand side of Figure~\ref{fig:3D_path}. Again, the stochastic properties resemble those of the observed processes on the left-hand side of Figure~\ref{fig:3D_path}.

\begin{figure}[tb]
	\centering
	\includegraphics[width=.52\textwidth]{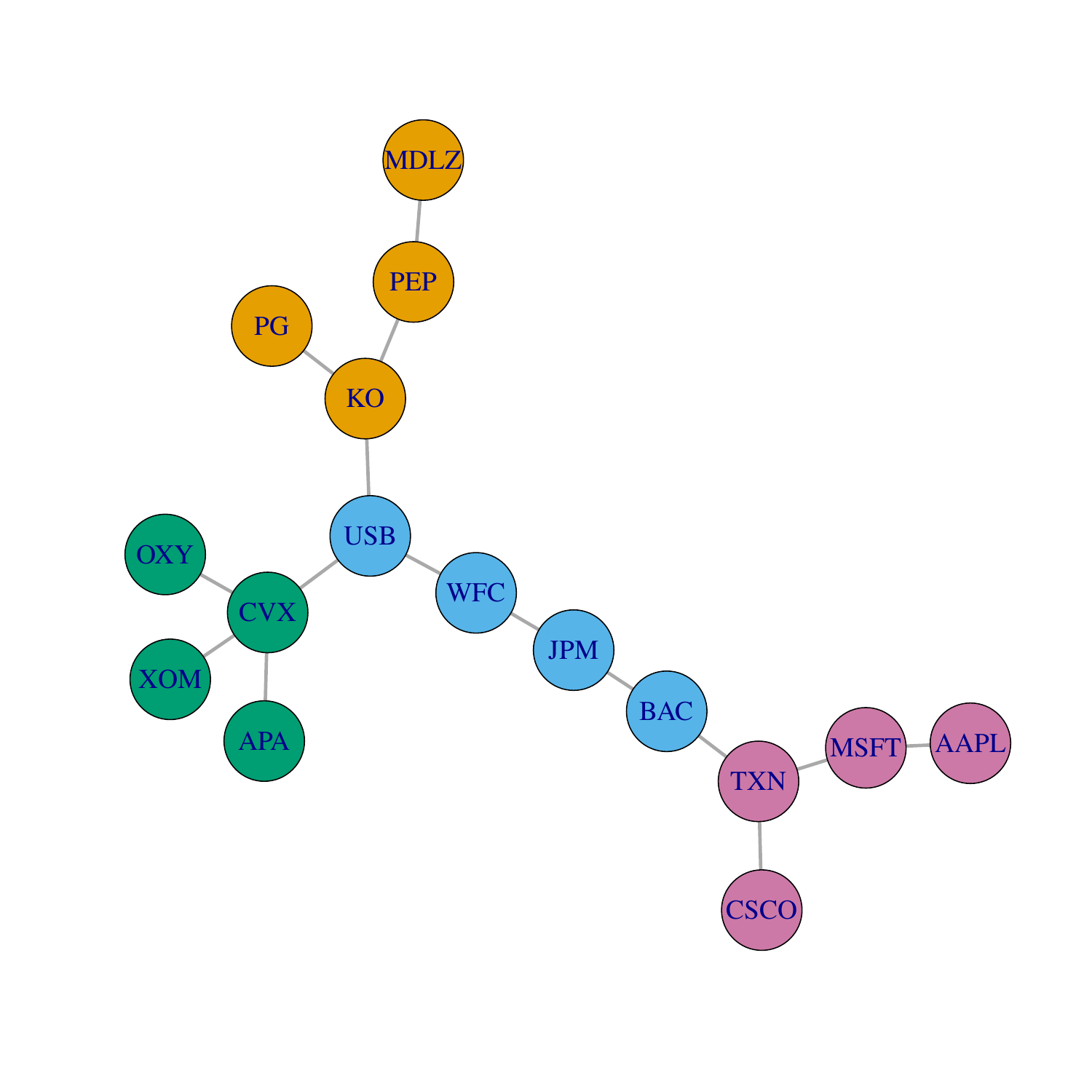} \hspace*{-3em}
	\includegraphics[width=.52\textwidth]{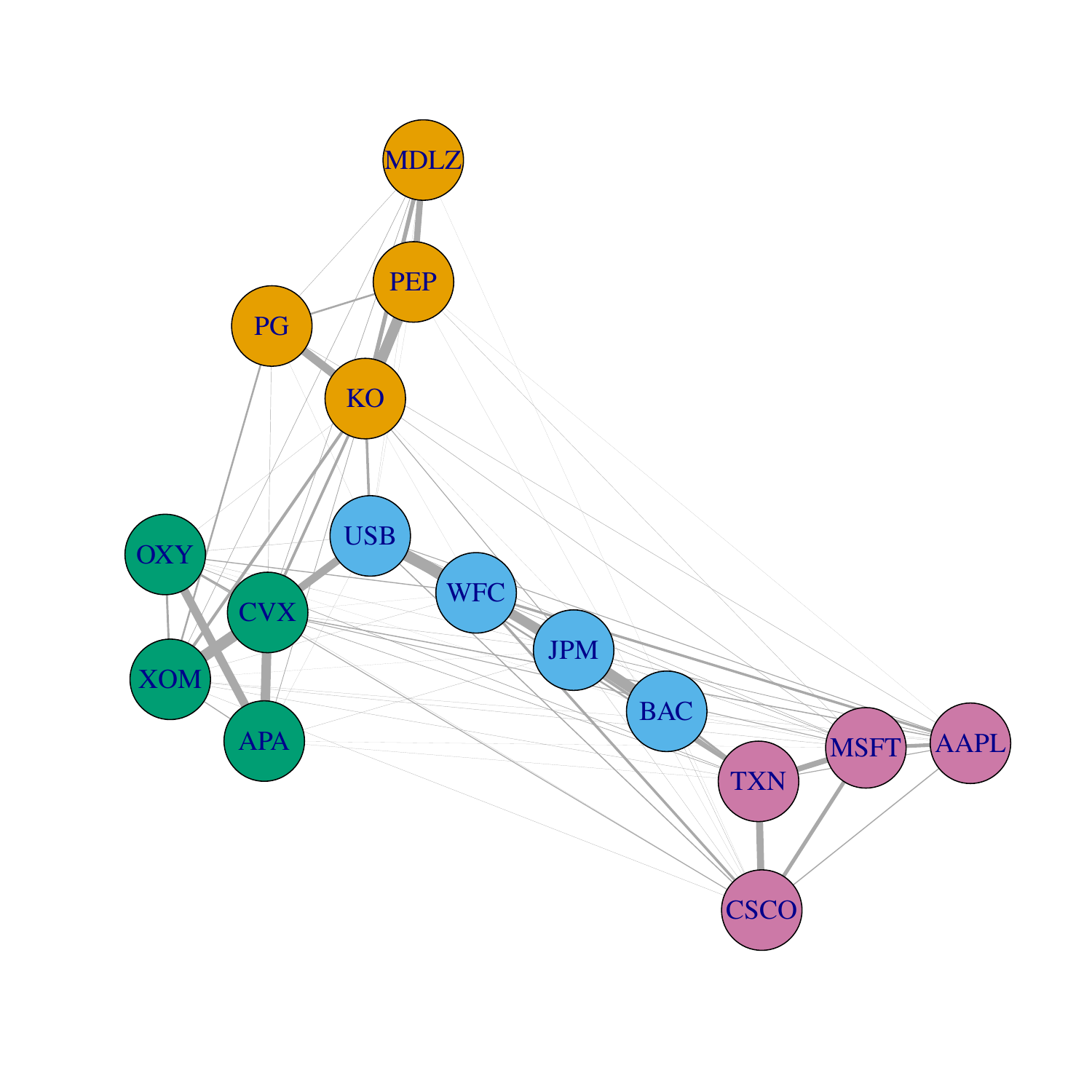}
	\caption{Left: Estimated minimum spanning tree $\widehat T$ for the stock price data. Right: Results for 300 estimated trees on subsampled data; the width of each edge indicates the number of times this edge was in the estimated tree. Colors indicate the corresponding industry sectors of the companies: Financials (yellow), Consumer Staples (blue), Technology (purple), Energy (green).}
	\label{fig:est_tree}
\end{figure}

\begin{table}[tb]
	\centering
	\begin{tabular}{|l|l|l|}
		\hline
		\textbf{Ticker} & \textbf{Company Name}  & \textbf{Industry Sector}                          \\ \hline
		JPM             & JPMorgan Chase         & Financials                              \\ \hline
		WFC             & Wells Fargo            & Financials                            \\ \hline
		BAC             & Bank of America        & Financials                              \\ \hline
		USB             & U.S. Bancorp           & Financials                               \\ \hline
		KO              & The Coca-Cola Company  & Consumer Staples                       \\ \hline
		PEP             & PepsiCo, Inc.          & Consumer Staples          \\ \hline
		MDLZ            & Mondelez International & Consumer Staples              \\ \hline
		PG              & Procter \& Gamble      & Consumer Staples  \\ \hline
		CVX             & Chevron Corporation    & Energy                               \\ \hline
		OXY             & Occidental Petroleum   & Energy                          \\ \hline
		XOM             & Exxon Mobil            & Energy                              \\ \hline
		APA             & APA Corporation        & Energy                        \\ \hline
		AAPL            & Apple Inc.             & Technology    \\ \hline
		MSFT            & Microsoft Corporation  & Technology        \\ \hline
		CSCO            & Cisco Systems          & Technology    \\ \hline
		TXN             & Texas Instruments      & Technology                   \\ \hline
	\end{tabular}
	\caption{Company names and industry sectors for selected tickers in the data.}
	\label{table:companies}
\end{table}

\section{Conclusion}

In this paper we introduce graphical models for L\'evy processes and characterize the conditional independence properties in terms of the L\'evy measure. Our theory opens the door to a rich mathematical theory and efficient statistical methodology for L\'evy processes. 

We consider the special case of tree graphs to obtain sparse statistical models. Many extensions of our work in terms theory and methodology are possible. For instance, in our asymptotic theory we rely on the heterogenous stable model to ensure that information on the L\'evy measure can be obtained from observations of the L\'evy process $\mathbf X$ at fixed distances in time. In the high-frequency setting with decreasing inter-observations times the assumption of marginal stability can potentially be dropped \citep[e.g.,][]{bas1982}.
Even for for the low-frequency setting with observations at fixed distances, there are methods based on minimum-distance estimators to go beyond stability \citep{reiss2009}.

Many connections to recent work on sparse extreme value theory can be established. More general L\'evy graphical models than trees could be learned through regularized methods as \cite{engelke2022a}, \cite{wan2023graphical} and \cite{lederer2023extremes}. Moreover, positive dependence between the marginals of the L\'evy process \citep{REZ2023,roettgerSchmitz2023} or statistical inference in the presence of latent processes \citep{engelke2024extremal} could be studied.

\appendix

\section{Proofs and other technical details}

\subsection{Decompositions and conditioning}\label{sec:decomp_cond}

The L\'evy process $\mathbf{X}$ may be represented by its L\'evy--Itô decomposition $\mathbf{X}=\mathbf{J}+\mathbf{W}$, where $\mathbf{J}$ is a L\'evy process with characteristic triplet $(0,0,\Lambda)$, and $\mathbf{W}$ is a Brownian motion with drift $\gamma$ and covariance matrix $\Sigma$ which is independent of $\mathbf{J}$. Note that $\mathbf{J}$ contains all the jumps of $\mathbf{X}$. For $C\subseteq\Set{1,\dotsc,d}$ the processes $\mathbf{J}_C,\mathbf{W}_C$ may be obtained as almost sure limits of processes $\mathbf{J}_C\ph{(n)},\mathbf{W}_C\ph{(n)}$ created from $\mathbf{X}_C$ as $n\to\infty$ \citep[e.g.,][\S2.4]{applebaum}. Thus, $\mathbf{J}_C,\mathbf{W}_C$ are $\bar\sigma(\mathbf{X}_C)$-measurable, where $\bar\sigma(\mathbf{X}_C)$ denotes the $\p$-completion of $\sigma(\mathbf{X}_C)$. For $F\in\sigma(\mathbf{J}_C,\mathbf{W}_C)$ we may therefore write $F=F'\cup N$, where $N$ is a null-set and $F'\in\sigma(\mathbf{X}_C)$. For an integrable random variable $Z$ one immediately finds that
\begin{equation*}
\int_F {\Mean{Z\given\mathbf{X}_C}} \mathrm d\mathbb P = \int_{F'} {\Mean{Z\given\mathbf{X}_C}} \mathrm d\mathbb P= \int_{F'} Z \mathrm d\mathbb P =\int_{F} Z \mathrm d\mathbb P.
\end{equation*}
Since $\Mean{Z\given\mathbf{X}_C}$ is $\sigma(\mathbf{J}_C,\mathbf{W}_C)$-measurable it follows that
\begin{equation*}
\Mean{Z\given\mathbf{X}_C}=\Mean{Z\given\mathbf{J}_C,\mathbf{W}_C}
\end{equation*}
almost surely.

In addition to the L\'evy--Itô decomposition we may represent $\mathbf{X}$ in other ways. Suppose for simplicity that $\mathbf{X}$ is a L\'evy process with triplet $(\gamma,0,\Lambda)$. Define $\Lambda\ph{(1)}=\Lambda(\cdot\cap\Set{x_C=0_C})$ and $\Lambda\ph{(2)}=\Lambda(\cdot\cap\Set{x_C\neq0_C})$. Then $\mathbf{X}$ may be represented as the independent sum of two L\'evy processes $\mathbf{X}\ph{(1)},\mathbf{X}\ph{(2)}$ with triplets $(\gamma,0,\Lambda\ph{(1)})$ and $(0,0,\Lambda\ph{(2)})$. The processes $\mathbf{X}_C\ph{(1)},\mathbf{X}_C\ph{(2)}$ are constructed from $\mathbf{X}_C$ and with arguments similar to the above we find that
\begin{equation*}
\Mean{Z\given\mathbf{X}_C}=\Mean{Z\given\mathbf{X}_C\ph{(1)},\mathbf{X}_C\ph{(2)}}
\end{equation*}
almost surely for any integrable random variable $Z$.

\subsection{Proof of Proposition~\ref{prop:fixed_times}}\label{proof:fixed_times}

\begin{proof}
Assume that $X_A(t)\Perp X_B(t)\mid X_C(t)$ for all $t\geq0$. We need to show that
\begin{equation}\label{eq:cond_fdd}
\Set{X_A(t_i)}_{i=1}^k\Perp\Set{X_B(t_i)}_{i=1}^k\mid\mathbf{X}_C
\end{equation}
for any $k\in\N$ and $0=t_1<\dotsc<t_k$.

For each $n\geq k$ we let $S\ph{(n)}=\Set{s\ph{(n)}_j}_{j=1}^n$ be a collection of time points such that
\begin{itemize}
\item For any $n\geq k$ and $j=1,\dotsc,n-1$ we have $0\leq s\ph{(n)}_j<s\ph{(n)}_{j+1}\leq t_k$.
\item For any $n\geq k$ and $i=1,\dotsc,k$ we have $t_i\in S\ph{(n)}$.
\item The sequence of sets $(S\ph{(n)})$ is increasing.
\item There is the convergence $\max_{j=1,\dotsc,n-1}(s\ph{(n)}_{j+1}-s\ph{(n)}_j)\to0$.
\end{itemize}

Now, for each $n\geq k$ and  $j=1,\dotsc,n-1$ we let $I\ph{(n)}(j)=X(s\ph{(n)}_{j+1})-X(s\ph{(n)}_j)$ denote the increment of $\mathbf{X}$ between times $s\ph{(n)}_j$ and $s\ph{(n)}_{j+1}$. Using stationarity and independence of the increments along with the initial conditional independence assumption we deduce that
\begin{equation*}
\Set{I\ph{(n)}_A(j)}_{j=1}^{n-1}\Perp\Set{I\ph{(n)}_B(j)}_{j=1}^{n-1}\mid\Set{I\ph{(n)}_C(j)}_{j=1}^{n-1}.
\end{equation*}
From the increments we may construct $\Set{X_A(t_i)}_{i=1}^k$ and $\Set{X_B(t_i)}_{i=1}^k$, and we therefore have that
\begin{equation*}
\Set{X_A(t_i)}_{i=1}^k\Perp\Set{X_B(t_i)}_{i=1}^k\mid\Set{I\ph{(n)}_C(j)}_{j=1}^{n-1}\quad\text{for all}\quad n\geq k.
\end{equation*}

It is sufficient to prove \eqref{eq:cond_fdd} with $\mathbf{X}_C$ replaced by $(X_C(s))_{0\leq s\leq t_k}$ since conditioning on the process $(X_C(t_k+s)-X_C(t_k))_{s\geq0}$ is added using independence. By the above conditional independence it is enough to show the identity
\begin{equation*}
\sigma(X_C(s),0\leq s\leq t_k)=\G,
\end{equation*}
where $\G=\cup_{n\geq k}\sigma(I_C\ph{(n)}(j),j=1,\dotsc,n-1)$. The inclusion $\supseteq$ is obvious. For the other inclusion it suffices to show that $X_C(s)$ is the limit of a sequence of $\G$-measurable random vectors for any $s\in[0,t_k]$. For each $n\geq k$ we pick an index $j_n$ such that the sequence $(s\ph{(n)}_{j_n})$ converges to $s$ from the right. Then $X_C(s\ph{(n)}_{j_n})$ converges to $X_C(s)$, and since each $X_C(s\ph{(n)}_{j_n})$ is $\G$-measurable, this concludes the proof.
\end{proof}

\subsection{Proof of Proposition~\ref{Brownian_CI}}\label{proof:Brownian_CI}

\begin{proof}
	Proposition~\ref{prop:fixed_times} implies the direction from right to left. Since the Brownian motion is a stable process, it suffices to have the conditional independence for one $t_0>0$, as discussed above.

	For the other direction, we show that conditional independence statement for a Wiener process $\mathbf W_A \Perp \mathbf W_B \mid \mathbf W_C$ implies 
	the respective conditional independence at a fixed positive time, say 1: $W_A(1) \Perp W_B(1) \mid W_C(1)$. 
	We start from
	\[W_A(1) \Perp W_B(1) \mid \mathbf W_C, \]
	and observe that the conditioning variable can be reduced to $(W_C(t): t\in [0,1])$ by independence of increments.
	Furthermore, it can be replaced by the pair $W_C(1), (B_C(t): t\in [0,1])$ with $B(t) := W(t)-tW(1)$ being a Brownian bridge. It is well-known and easy to see that $W(1) \Perp (B(t): t\in [0,1])$ and so the bridge can be dropped from the conditioning set yielding the required statement. 
\end{proof}

\subsection{Proof of Proposition~\ref{prop:killed_process}}
\begin{proof}
For any $T>0$ we note that conditioning on $\mathbf{X}_C$ is the same as conditioning on the pair $\Set{\mathbf{X}_C^T,(X_C(T+t)-X_C(T))_{t\geq0}}$. Here we use the fact that $\mathbf{X}$ is a.s.\ continuous at time $T$.

Assume that $\mathbf{X}_A\Perp\mathbf{X}_B\mid\mathbf{X}_C$ and let $T>0$. Then $\mathbf{X}_A^T\Perp\mathbf{X}_B^T\mid\mathbf{X}_C$. Since $\mathbf{X}^T$ is independent of $(X_C(T+t)-X_C(T))_{t\geq0}$ we also have conditional independence given just $\mathbf{X}_C^T$.

For the opposite implication we assume that $\mathbf{X}_A^T\Perp\mathbf{X}_B^T\mid\mathbf{X}_C^T$ for all $T>0$. For $k\in\N$ and $0\leq t_1<\dotsc<t_k$ we pick $T>t_k$ and note that $X_A(t_i)=X_A^T(t_i)$ for all $i=1,\dotsc,k$ (and similarly for the $B$-component). Hence, $\Set{X_A(t_i)}_{i=1}^k\Perp\Set{X_B(t_i)}_{i=1}^k\mid\mathbf{X}_C^T$. Finally we employ independence to further condition on the process $(X_C(T+t)-X_C(T))_{t\geq0}$. 
\end{proof}

\subsection{Proof of Proposition~\ref{prop:add_gaussian}}\label{proof:add_gaussian}

\begin{lemma}\label{lem:cond_indep_limit}
Let $(X_n),(Y_n)$ be sequences of random variable defined on a probability space~$(\Omega,\ff,\p)$, taking values in Polish spaces $S_X,S_Y$, and assume that there exist random variables $X,Y,Z$ such that $X_n\to X$ a.s., $Y_n\to Y$ a.s.\ and $X_n\Perp Y_n\mid Z$ for all $n\geq1$. Then $X\Perp Y\mid Z$.
\end{lemma}
\begin{proof}
The collection of random variables $((X_n),(Y_n),X,Y)$ takes values in a Polish space, thus establishing the existence of a regular conditional distribution given $Z$. Denoting this probability kernel by $\mu$ we note that $\p$-almost surely
\begin{equation*}
\mu(Z,\Set{x_n\in A,y_n\in B})=\mu(Z,\Set{x_n\in A})\mu(Z,\Set{y_n\in B})
\end{equation*}
for all $A\in\B(S_X),B\in\B(S_Y)$, where, e.g., $\Set{x_n\in A}=\Set{((x_m),(y_m),x,y)\in S_X^\N\times S_Y^\N\times S_X\times S_Y\given x_n\in A}$. Furthermore, $\mu(Z,\Set{x_n\to x,y_n\to y})=1$ almost surely; see properties of kernels in \citet[][Chapter 8]{kallenberg3}.
Now, since independence is preserved under almost sure convergence the result follows.
\end{proof}

\begin{proof}[Proof of Proposition~\ref{prop:add_gaussian}]
Assume that $\mathbf{J}_A\Perp\mathbf{J}_B\mid\mathbf{J}_C$ and $\mathbf{W}_A\Perp\mathbf{W}_B\mid\mathbf{W}_C$. Recall from Appendix~\ref{sec:decomp_cond} that conditioning on $\mathbf{X}_C$ is the same as conditioning on $(\mathbf{J}_C,\mathbf{W}_C)$. We combine this with the independence of $\mathbf{J}$ and $\mathbf{W}$ to obtain the identity
\begin{align*}
\MoveEqLeft
\Mean{\ind{E_A\times F_A}(\mathbf{J}_A,\mathbf{W}_A)\ind{E_B\times F_B}(\mathbf{J}_B,\mathbf{W}_B)\given\mathbf{X}_C} \\
&=\Mean{\ind{E_A\times E_B}(\mathbf{J}_A,\mathbf{J}_B)\given\mathbf{J}_C}\Mean{\ind{F_A\times F_B}(\mathbf{W}_A,\mathbf{W}_B)\given\mathbf{W}_C}
\end{align*}
for any Borel sets $E_A,E_B,F_A,F_B$.
The conditional expectations on the right-hand side factorize due to the assumed conditional independence. Applying the above identity twice with appropriately chosen $E_A,E_B,F_A,F_B$ yields
\begin{align*}
\MoveEqLeft
\Mean{\ind{E_A\times F_A}(\mathbf{J}_A,\mathbf{W}_A)\ind{E_B\times F_B}(\mathbf{J}_B,\mathbf{W}_B)\given\mathbf{X}_C} \\
&=\Mean{\ind{E_A\times F_A}(\mathbf{J}_A,\mathbf{W}_A)\given\mathbf{X}_C}\Mean{\ind{E_B\times F_B}(\mathbf{J}_B,\mathbf{W}_B)\given\mathbf{X}_C}.
\end{align*}
Hence, the pairs $(\mathbf{J}_A,\mathbf{W}_A)$ and $(\mathbf{J}_B,\mathbf{W}_B)$ are conditionally independent given $\mathbf{X}_C$. The stated conditional independence follows immediately.

Assume instead that $\mathbf{X}_A\Perp\mathbf{X}_B\mid\mathbf{X}_C$. The processes $\mathbf{J}_A,\mathbf{W}_A$ can be constructed as almost sure limits of functions of $\mathbf{X}_A$, and similarly for the $B$-component. According to Lemma~\ref{lem:cond_indep_limit} we then have
\begin{equation*}
\mathbf{J}_A\Perp\mathbf{J}_B\mid\mathbf{X}_C\qtq{and}\mathbf{W}_A\Perp\mathbf{W}_B\mid\mathbf{X}_C.
\end{equation*}
As previously discussed, conditioning on $\mathbf{X}_C$ is the same as conditioning on the pair $(\mathbf{J}_C,\mathbf{W}_C)$. Now the claimed conditional independence follows since $\mathbf{J}$ and $\mathbf{W}$ are independent.
\end{proof}

\subsection{Proof of Theorem~\ref{thm:cond_indep_process_level}}\label{proof:cond_indep_process_level}

\begin{proof}
\emph{Part 1:} Assume that $\mathbf{X}_A\Perp\mathbf{X}_B\mid\mathbf{X}_C$. Let $c\in C$ and $\epsilon>0$ We may write $\mathbf{X}$ as the independent sum $\mathbf{X}=\mathbf{X}^{\geq\epsilon}+\mathbf{X}^{<\epsilon}$, where $\mathbf{X}^{\geq\epsilon}$ is a compound Poisson process consisting of the jumps of $\mathbf{X}$ in $\Set{\abs{x_c}\geq\epsilon}$. Then $\mathbf{X}^{\geq\epsilon}_A$ is a function of $\mathbf{X}_A,\mathbf{X}_C$ and similarly for $\mathbf{X}^{\geq\epsilon}_B$. It follows that $\mathbf{X}^{\geq\epsilon}_A\Perp\mathbf{X}^{\geq\epsilon}_B\mid\mathbf{X}_C$. Since $\mathbf{X}^{\geq\epsilon}_C,\mathbf{X}^{<\epsilon}_C$ are functions of $\mathbf{X}_C$ and \textit{vice versa}, we have that 
that the generated $\sigma$-algebras coincide $\sigma(\mathbf X_C) = \sigma(\mathbf{X}^{\geq\epsilon}_C,\mathbf{X}^{<\epsilon}_C)$. This implies that 
$\mathbf{X}^{\geq\epsilon}_A\Perp\mathbf{X}^{\geq\epsilon}_B\mid\mathbf{X}^{\geq\epsilon}_C,\mathbf{X}^{<\epsilon}_C$, which is the same as $\mathbf{X}^{\geq\epsilon}_A\Perp\mathbf{X}^{\geq\epsilon}_B\mid\mathbf{X}^{\geq\epsilon}_C$ by independence. We may write
\begin{equation*}
X^{\geq\epsilon}(t)=\sum_{n=1}^{N(t)}Y(n),\For{$t\geq0$,}
\end{equation*}
where $\mathbf{N}$ is a Poisson process with rate $\Lambda({\abs{x_c}\geq\epsilon})$ and $(Y(n))$ is a sequence of i.i.d.\ random vectors independent of $\mathbf{N}$ with $Y(1)\sim\p_{c,\epsilon}=\Lambda(\cdot\cap\Set{\abs{x_c}\geq\epsilon})/\Lambda({\abs{x_c}\geq\epsilon})$. Since $\mathbf{N}$ and $(Y(n))$ can be obtained from $\mathbf{X}^{\geq\epsilon}$ we find that $Y_A(1)\Perp Y_B(1)\mid Y_C(1)$. According to \citet[Theorem~4.1]{eng_iva_kir} it remains to prove that $A\perp B\,[\Lambda_{A\cup B}^0]$. We consider a different decomposition $\mathbf{X}=\mathbf{X}^{=0}+\mathbf{X}^{\neq0}$, where $\mathbf{X}^{=0},\mathbf{X}^{\neq0}$ are independent L\'evy processes with L\'evy measures given by $\Lambda$ restricted to $\Set{x_C=0_C}$ and $\Set{x_C\neq0_C}$ respectively. The process $\mathbf{X}^{=0}$ can be constructed as an almost sure limit of $\sigma(\mathbf{X})$-measurable processes (we refer to the discussion in Appendix~\ref{sec:decomp_cond}). By applying Lemma~\ref{lem:cond_indep_limit} we find that $\mathbf{X}^{=0}_A\Perp \mathbf{X}^{=0}_B\mid\mathbf{X}^{=0}_C$, and in fact we have $\mathbf{X}^{=0}_A\Perp \mathbf{X}^{=0}_B$ since $\mathbf{X}^{=0}_C$ is deterministic. Then the independence $A\perp B\,[\Lambda_{A\cup B}^0]$ follows from Lemma~\ref{lem:independence_implication} as we note that $\mathbf{X}^{=0}_{A\cup B}$ has L\'evy measure $\Lambda_{A\cup B}^0$.\\

\emph{Part 2:} Assume that $A\perp B\mid C\,[\Lambda]$ and $C\neq \emptyset$. First we consider the case where $\Lambda$ is a finite L\'evy measure such that $\Lambda({x_C=0_C})=0$. Then $\mathbf{X}$ is the sum of a linear drift and a compound Poisson process. To be precise,
\begin{equation*}
X(t)=t\gamma+\sum_{n=1}^{N(t)}Y(n),\For{$t\geq0$,}
\end{equation*}
where $\mathbf{N}$ is a Poisson process with rate $\Lambda(\R^d)$ and $(Y(n))$ is a sequence of i.i.d.\ random variables, independent of $\mathbf{N}$ and with $Y(1)\sim\Lambda(\cdot)/\Lambda(\R^d)$. We note that conditioning on $\mathbf{X}_C$ will also fix~$\mathbf{N}$ since $\Lambda({x_C=0_C})=0$ and there are thus no jumps only in $\mathbf{X}_A$ and $\mathbf{X}_B$ but not in $\mathbf{X}_C$. For each $n\in\N$ we further find that the random variables $Y_A(n)$ and $Y_B(n)$ are conditionally independent given $\mathbf{X}_C$. Indeed, knowledge of $\mathbf{X}_C$ also implies knowledge of $Y_C(n)$, and by the assumption $A\perp B\mid C\,[\Lambda]$ we obtain $Y_A(n) \Perp Y_B(n)\mid Y_C(n)$. Now we apply the ideas from the proof of Proposition~\ref{prop:add_gaussian} to conclude that $\mathbf{X}_A\Perp \mathbf{X}_B\mid\mathbf{X}_C$.

Now we no longer assume that $\Lambda$ is finite. Instead we assume that $\Lambda({x_c=0})=0$ for some $c\in C$. We may view $\mathbf{X}$ as an almost sure limit of a sequence $(\mathbf{X}\ph{(n)})$ of L\'evy processes, where $(\mathbf{X}\ph{(n)})$ has \emph{finite} L\'evy measure given by $\Lambda\ph{(n)}=\Lambda(\cdot\cap\Set{1/n<\abs{x_c}})$. Since $A\perp B\mid C\,[\Lambda]$ we also have $A\perp B\mid C\,[\Lambda\ph{(n)}]$ for every $n\in\N$. Then, noting that $\Lambda({x_C=0_C})\leq \Lambda({x_c=0})=0$, by the previous paragraph, we get that $\mathbf{X}\ph{(n)}_A\Perp \mathbf{X}\ph{(n)}_B\mid\mathbf{X}_C$ since conditioning on $\mathbf{X}_C$ or $\mathbf{X}\ph{(n)}_C$ has the same effect on $\mathbf{X}\ph{(n)}$. Using Lemma~\ref{lem:cond_indep_limit} the conditional independence $\mathbf{X}_A\Perp \mathbf{X}_B\mid\mathbf{X}_C$ follows.

Finally we can consider a general L\'evy measure $\Lambda$. We proceed with induction in $k=\abs{C}$. From Lemma~\ref{lem:independence_implication} we know that the result holds for $k=0$. Now, assume that it holds when $C$ has $k-1$ elements, and fix some $c\in C$. We write $\mathbf{X}=\mathbf{X}'+\mathbf{X}''$, where $\mathbf{X}'$ and $\mathbf{X}''$ are independent L\'evy processes with L\'evy measures $\Lambda'=\Lambda(\cdot\cap\Set{x_c=0})$ and $\Lambda''=\Lambda(\cdot\cap\Set{x_c\neq0})$. We have $A\perp B\mid C\,[\Lambda']$ and since $\Lambda'$ is concentrated on $\Set{x_c=0}$ it follows that $A\perp B\mid C\setminus\Set{c}\,[\Lambda'_{V\setminus\Set{c}}]$. Importantly, $\Lambda'_{V\setminus\Set{c}}$ satisfies \ref{eq:explosion_1}, so by the induction hypothesis we have that $\mathbf{X}'_A\Perp\mathbf{X}'_B\mid\mathbf{X}'_{C\setminus\Set{c}}$. We note that this still holds if we condition on all of $\mathbf{X}'_C$ instead because $\mathbf{X}'_c$ is deterministic. We further have $A\perp B\mid C\,[\Lambda'']$, and since $\Lambda''({x_c=0})=0$ we have $\mathbf{X}''_A\Perp\mathbf{X}''_B\mid\mathbf{X}''_C$ by the paragraph above. To conclude we make use of the fact that conditioning on $\mathbf{X}_C$ is the same as conditioning on the pair $(\mathbf{X}'_C,\mathbf{X}''_C)$ as discussed in Appendix~\ref{sec:decomp_cond}.
By combining this with the conditional independence statements for the two terms we arrive at the conditional independence $\mathbf{X}_A\Perp \mathbf{X}_B\mid\mathbf{X}_C$. For the last step we again use the ideas from the proof of Proposition~\ref{prop:add_gaussian}.
\end{proof}

\subsection{Proof of Proposition~\ref{prop_copula}}

\begin{proof}
	First, we provide a proof for the case when $\Lambda$ is supported by the positive orthant.
	For each $i\in V$ consider the transformation $T_i(x)=U^{-1}_i(U'_i(x))$, 
	$x\in [0,\infty]$, where 
	\[U^{-1}_i(a)=\inf\{b\geq 0 \mid U_i(b)=a\}\] is the left-continuous inverse of $U_i$ which is well-defined since $U_i$ is continuous and explodes at $0^+$. Thus, for $x\geq 0$,
	\begin{align}		
		\notag\Lambda(y: y_i \geq T_i(x_i) \text{ for } i\in V) &= \Lambda^*(y: y_i \geq 1/U'_i(x_i) \text{ for } i\in V) \\ 
		\label{eq:UUprime}&= \Lambda'(y: y_i \geq x_i \text{ for } i\in V)
	\end{align}

	Furthermore, $T_i$ enjoys the following properties:
	\begin{itemize}
	\item[(a)] $T_i$ is monotonically increasing and is strictly monotone on the support of $\Lambda'_i$;
	\item[(b)] $T_i(0)=0$;
	\item[(c)] the range of $T_i$ is the support of $\Lambda_i$ restricted to $\R_+$.
	\end{itemize}
	
	For fixed $i\in V$ and $\epsilon_i>0$ define $\epsilon'_i=\inf\{\varepsilon>0: T_i(\varepsilon)\geq \epsilon_i\}$. Note that $T_i(\epsilon'_i)\leq \epsilon_i$ and the inequality is strict unless $U_i$ is constant preceeding $\epsilon_i$.
	With this definition we have
	\begin{align*}
		\Lambda(y: y_i \geq \epsilon_i, y_j \geq 0 \text{ for } j\in V\setminus\{i\}) &= \Lambda(y: y_i \geq T_i(\epsilon'_i), y_j \geq 0 \text{ for } j\in V\setminus\{i\}) \\ 
		&= \Lambda'(y: y_i \geq \epsilon'_i, y_j \geq 0 \text{ for } j\in V\setminus\{i\})  =:k. 
	\end{align*}
	The first equality follows from the absence of mass of $\Lambda_i$ in $[\epsilon_i, T_i(\epsilon'_i))$, and the second equality follows from~\eqref{eq:UUprime} and property~(b).
	
	Assume $k>0$ and let a random vector $X$ be distributed according to the law $\mathbb P(X \in A) = \Lambda(y: y \in A)/k$ for all Borel $A\subset \{x\geq (0,\ldots, \epsilon_i,\ldots, 0) \}$.
	Define the law of $X'$ analogously by taking $\Lambda'$ and $\epsilon'_i$ instead of $\Lambda$ and $\epsilon$.
	Now we have 
	\begin{align}&\p(X_1\geq T_1(x_1),\dots, X_d\geq T_d(x_d))=		
		\Lambda(y: y_i \geq T_i(x_i) \text{ for } i\in V)/k
		=\Lambda'(y: y_i \geq x_i \text{ for } i\in V)/k\label{eq:XXprime}\\
	&=\p(X'_1\geq x_1,\dots,X'_d\geq x_d)=\p(T_1(X'_1)\geq T_1(x_1),\dots,T_d(X'_d)\geq T_d(x_d)),\notag
	\end{align}
	where in the last equality we used property~(a).
	By property (c) we find that $X$ and $(T_i(X'_i))_i$ have the same distribution.
	
	Suppose that $A \indep B \mid C\; [\Lambda']$. Then for any $i\in V$ and $\epsilon>0$ we take random vectors $X$ and $X'$ as above. We have $X'_A \Perp X'_A \mid X'_C$ by definition of $A \indep B \mid C\; [\Lambda']$. But $T_i$ are injective on the support of the law of $X'$ and so $X_A \Perp X_A \mid X_C$. Now $A \indep B \mid C\; [\Lambda]$ follows from~\citet[Theorem~4.1]{eng_iva_kir}. This proves our result in the case of $\Lambda$ being supported by the positive orthant.
	
	The general case follows along the same steps, but is more cumbersome in terms of notation.
	The main difference is that the laws of vectors $X$ and $X'$ are defined on the sets $\{x\in \R^d: |x_i|\geq \epsilon_i\}$ and $\{x\in \R^d: x_i\geq \epsilon'_i\text{ or }x_i\leq -\widehat\epsilon'_i\}$, respectively, where $\widehat\epsilon'_i$ is defined analogously to $\epsilon'_i$ but using the transformation $\widehat T_i$ arising from the $i$th marginals on the negative half-line. To do so we simply combine tail functions for all quadrants. 
	Consider the analogues of~\eqref{eq:XXprime} for each quadrant, and establish distributional equality of $X$ and the transformed $X'$ (the sets used in all tail functions are measure determining for the considered domain). The rest of the proof is identical. 
	\end{proof}

\subsection{Proof of Proposition~\ref{prop_cop_same}}

\begin{proof}

	For the first assertion, note that if $\Lambda_D^0(\Set{x_v\neq0})>0$ for some $D\subseteq V$ and $v\in V$, then there is $\varepsilon>0$ with $\Lambda_D^0(\Set{\abs{x_v}\geq \varepsilon})>0$. Homogeneity of $\Lambda^*$ then implies that
\begin{align*}
	\Lambda_D^0(\Set{x_v\neq0})&= \lim_{u \to 0}\Lambda_D(\Set{\abs{x_v}> u\varepsilon, x_{V\setminus D}=0_{V\setminus D}}) \\
	& =  \lim_{u \to 0}\Lambda_D^*(x: \pm x_v >  \{c^\pm_v (u\varepsilon)^{-\alpha_v}\}^{-1}, x_{V\setminus D}=0_{V\setminus D}) \\	
	& =  \lim_{u \to 0}  u^{-\alpha_v} \Lambda_D^*(x: \pm x_v >  \{c^\pm_v \varepsilon^{-\alpha_v}\}^{-1}, x_{V\setminus D}=0_{V\setminus D}) \\	
	&= \lim_{u \to 0} u^{-1} \Lambda_D^0(\Set{\abs{x_v} > \varepsilon}) = \infty.
\end{align*}

For the time-invariance of the copula it is sufficient to  show that~\eqref{ss_hetero}, because the copula is preserved under strictly increasing transformations of the marginals. From~\eqref{levy_std1} and homogeneity of $\Lambda^*$ observe that 
\begin{align*}t\Lambda(y:\pm y_1> x_1,\dots,\pm y_d>x_d) &= \Lambda^*\left(y: \pm y_1 > (c^\pm_1tx^{-\alpha_1})^{-1},\dots,\pm y_d > (c^\pm_dtx^{-\alpha_d})^{-1}\right)\\
&=\Lambda(y:\pm y_1> t^{-1/\alpha_1}x_1,\dots,\pm y_d>t^{-1/\alpha_d}x_d).\end{align*}
Denote the exponent of the characteristic function in the L\'evy--Khintchine formula by $t \psi(u) = \log \e{e^{i\inner{u}{X(t)}}}$, for $t\geq 0$. Perform the change of variables $x'_j=t^{-1/\alpha_j}x_j$ to get 
\[t\int_{\R^d}e^{i\inner{u}{x}}-1-i\inner{u}{x}\ind{\norm{x}\leq 1}\,\Lambda(\dd x)=
\int_{\R^d}e^{i\inner{u'}{x'}}-1-i\inner{u'}{x'}\ind{\norm{x'}\leq 1}\,\Lambda(\dd x') + i\inner{u'}{\gamma'(t)},\]
where $u'_j=t^{1/\alpha_j}u_j$ and $\gamma'(t)=\int_{\R^d}x'\left(\ind{\norm{x'}\leq 1}-\ind{\norm{x}\leq 1}\right)\,\Lambda(\dd x')$.
Thus we have
\[\e{e^{i\inner{u}{X(t)}}}=e^{\psi(u)t}=e^{\psi(u')+ i\inner{u'}{\gamma'(t)}}=\e{e^{i\sum_j u_jt^{1/\alpha_j}(X_j(1)+\gamma'_j(t))}},\]
and so the claimed representation is proven with $\gamma_j(t)=t^{1/\alpha_j}\gamma'_j(t)$.
\end{proof}

\subsection{Proof of Lemma~\ref{lem_chi}}\label{proof:lem_chi}

\begin{proof}
	We first note that for some $p\in [0,1]$ and some $q>0$ we have that 
	\[ |2p -1| > q \quad \Leftrightarrow \quad p < (1-q)/2 \text{ or } p > (1+ q)/2 .\]
	Therefore, the probability on the right-hand side of~\eqref{chi_rep} measures the event that the $i$th and $j$th components, on copula scale, are simultaneously either below $(1-q)/2$ or above $(1+ q)/2$.

	The approximation of the standardized L\'evy measure $\Lambda^*$ by the copulas of the process $X(t)$ at small times $t$ in~\eqref{cop_approx} yields
	\begin{align*}
		\Lambda^*(x: x_i >u_i, x_j >u_j) = \lim_{t\to 0} t^{-1} \mathbb P\left\{F_{ti}(X_i(t)) > 1 - t/u_i, F_{tj}(X_j(t)) > 1 - t/u_j \right\}, \quad u_i,u_j >0,
	\end{align*}
	and similarly for the other orthants. By Proposition~\ref{prop_cop_same}, the copulas of $X(t)$ coincide for all $t>0$ for a L\'evy process $\mathbf X$ in the class of heterogenous stable models that we assume. Therefore,
	\begin{align*}
		\Lambda^*(x: x_i > 2,x_j > 2)= \lim_{q\to 1} (1-q)^{-1} \mathbb P\left\{F_{1i}(X_1(1)) > (1 + q)/2, F_{1j}(X_d(1)) > (1 + q)/2 \right\},
	\end{align*}
	and
	\begin{align*}
		\Lambda^*(x: x_i < -2,x_j > 2)= \lim_{q\to 1} (1-q)^{-1} \mathbb P\left\{F_{1i}(X_1(1)) < (1-q)/2, F_{1j}(X_d(1)) > (1 + q)/2 \right\};
	\end{align*}
	and similarly for the cases $\{x_i \leq -2,x_j\leq -2\}$ and $\{x_i \geq 2,x_j\leq -2\}$. This shows that 
	\begin{align*}
		\lim_{q \to 1}  \frac{\mathbb P\{ |2F_{1i}(X_i(1)) - 1| > q, |2F_{1j}(X_j(1)) -1 | > q \}}{1-q}   &=   \Lambda^*(y: |y_i| > 2, |y_j| > 2)\\
		& =  \Lambda^*(y: |y_i| > 1, |y_j| > 1)/2\\
		& = \chi_{ij}, 
	\end{align*}
	where the second equation follows from the homogeneity of $\Lambda^*$.
\end{proof}

\subsection{Proof of Proposition~\ref{prop:chi_ineq}}\label{proof:prop_chi_ineq}

\begin{proof}
	We start by defining a positive version $\bar \Lambda^*$ of the normalized L\'evy measure $\Lambda^*$ as
	\[\bar \Lambda^*(x: x_1\geq u_1^{-1},\dots,x_d\geq u_d^{-1}) = \Lambda^*(x: |x_1| \geq u_1^{-1},\dots,|x_d|\geq u_d^{-1}), \quad u >0,\]
	and $\bar \Lambda^*$ has only mass on the positive orthant.
	It is easy to see that under the symmetry assumption~\eqref{levy_std1}, the new measure $\bar \Lambda^*$ is still a tree graphical model on $T$. We can now use existing results from extreme value theory literature. In particular, $\bar \Lambda^*$ is a valid exponent measure of a max-stable distribution. The inequality~\eqref{eq:chi_ineq} then follows directly from \citet[Proposition 5]{engelke-volgushev}. The same proposition also implies the identification of the underlying tree structure through the minimum spanning tree problem, given that the inequalities are strict. 
\end{proof}

\subsection{Proof of Proposition~\ref{prop:m_consistency}}
\begin{proof}
	By the same arguments as in the proof of Theorem~\ref{thm:recovery}, the estimators $\widehat \chi_{ij}^{++}$ and $\widehat \chi_{ij}^{--}$ are consistent for the true counterparts. An application of the continuous mapping theorem yields the consistency of $\widehat m_{ij}$.
	
\end{proof}

\subsection{Statement and proof of Proposition~\ref{prop_hf}}\label{proof:prop_hf}

\begin{proposition}\label{prop_hf}
	Under the assumptions of Theorem~\ref{thm:recovery} on the L\'evy process $\mathbf X$, the high-frequency estimator satisfies for all $i,j\in V$
	\[ \widehat\chi^{\text{HF}}_{ij} \stackrel{d}{=} \widehat\chi_{ij}.\]
\end{proposition}

\begin{proof}
The high-frequency estimator can be written as 
\begin{equation}\label{eq:chi_hat_hf}
	\widehat\chi^{\text{HF}}_{ij}= \frac{1}{k}\sum_{t=1}^n\indnew{\Set{|2\widehat{F}_{ni}(\Delta_i(t,n)) - 1| > 1 - k/n, |2\widehat{F}_{nj}(\Delta_j(t,n))-1|>1 - k/n}},
\end{equation}
where for any $i=1,\dots, d$ and $t=1,\dots, n$
\begin{equation*}
	\widehat{F}_{ni}(\Delta_i(t,n))=\frac{1}{n+1}\sum_{s=1}^n \indnew{\Set{\Delta_i(s,n)\leq \Delta_i(t,n)}}.
	\end{equation*}
It suffices to observe that 
\begin{align*}
	\indnew{\Set{\Delta_i(s,n)\leq \Delta_i(t,n)}} &\stackrel{d}{=} \indnew{\Set{ X_{is}(1/n) \leq X_{it}(1/n)}}\\
	& \stackrel{d}{=} \indnew{\Set{ n^{-1/\alpha_i} X_{is}(1) + \gamma_i(1/n) \leq n^{-1/\alpha_i} X_{it}(1) + \gamma_i(1/n)}} \\
	& \stackrel{d}{=} \indnew{\Set{\Delta_i(s)\leq \Delta_i(t)}},
\end{align*}
jointly in $i=1,\dots, d$ and $t=1,\dots, n$, because of the self-similarity property~\eqref{ss_hetero}, and thus
\begin{align*}
	\widehat{F}_{ni}(\Delta_i(t,n))  &\stackrel{d}{=} \widehat{F}_{i}(\Delta_i(t)).
\end{align*}
This yields that $\widehat\chi^{\text{HF}}_{ij} \stackrel{d}{=} \widehat\chi_{ij}$.

\end{proof}

\subsection{Proof of Theorem~\ref{thm:recovery}}\label{proof:thm:recovery}

\begin{proof}
	We first establish the consistency of the estimator
	$\widehat \chi_{ij}$, $i,j\in V$, and note that it can rewritten as 
	\begin{align}\label{chi_split}
		 \widehat \chi_{ij} = \widehat \chi^{++}_{ij} + \widehat \chi^{+-}_{ij} + \widehat \chi^{-+}_{ij} + \widehat \chi^{--}_{ij}, 
	\end{align}
	where each of the four estimators counts the points in one of the corners. For instance, we have
	\begin{equation*}
		\widehat\chi^{++}_{ij}= \frac{1}{k}\sum_{t=1}^n\indnew{\Set{\widehat{F}_i(\Delta_i(t)) > 1 - k/(2n), \widehat{F}_j(\Delta_j(t)) >1 - k/(2n)}},
	\end{equation*}
	and similarly for the others.	
	Since this estimator has the same form as the commonly used extremal correlation estimator in extreme value theory, we can then rely existing results. In view of Lemma~\ref{lem_chi}, it suffices to show that $d$-dimensional random vector $X(1)$ is multivariate regularly varying.

	The approximation of the standardized L\'evy measure $\Lambda^*$ by the copulas of the process $X(t)$ at small times $t$ in~\eqref{cop_approx} yields
	\begin{align*}
		\Lambda^*(x: x_1\geq u_1^{-1},\dots,x_d\geq u_d^{-1}) = \lim_{t\to 0} t^{-1} \mathbb P\left\{F_{t1}(X_1(t)) > 1 - tu_1, \dots, F_{td}(X_d(t)) > 1 - tu_d \right\}.
	\end{align*}
	By Proposition~\ref{prop_cop_same}, the copulas $C_t\equiv C_1$ coincide for all $t>0$ for a L\'evy process $\mathbf X$ in the class of heterogenous stable models that we assume. Therefore,
	\begin{align*}
		\Lambda^*(x: x_1\geq u_1^{-1},\dots,x_d\geq u_d^{-1})= \lim_{t\to 0} t^{-1} \mathbb P\left\{F_{11}(X_1(1)) > 1 - tu_1, \dots, F_{1d}(X_d(1)) > 1 - tu_d \right\}.
	\end{align*}
	The left-hand side is related to the $\ell$-function \citep[e.g.,][Chapter 6]{deh2006a}, and, indeed, this convergence is nothing else than the well-known assumption of multivariate regular variation in extreme value theory \citep[e.g.,][Theorem 2]{segers2020}.

	To summarize, we have established that $X(1)$ is a multivariate regularly varying random vector and $\widehat \chi_{ij}^{++}$ is has exactly the same form as the empirical extremal correlation estimator from extreme value theory. Under the additional standard assumption that $k\to\infty$ and $k/n\to 0$, as $n\to \infty$, the consistency of $\widehat \chi^{++}_{ij}$ follows, for instance, from \citet[Section 4.2]{engelke-volgushev}, that is, for any $\varepsilon >0$, we have
	\[ \lim_{n\to\infty} \mathbb P( | \widehat \chi^{++}_{ij} - \chi^{++}_{ij} |>\varepsilon) = 0,\]
	where $\chi^{++}_{ij} = \Lambda^*(y: y_i > 2, y_j > 2)$. Consistency of the other three estimators in~\eqref{chi_split}
	follows analogously, and thus	
	\[ \lim_{n\to\infty} \mathbb P( | \widehat \chi_{ij} - \chi_{ij} |>\varepsilon) = 0.\]
	Consistent tree recovery now follows immediately along the lines of the proof of \citet[Theorem 2]{engelke-volgushev}, who showed this result for extremal tree models.
\end{proof}

\section{Small-jump approximation}\label{small_jump_approx}

In this section we assume that the L\'evy process $\mathbf{X}$ is $\alpha$-stable with $\alpha\in(0,2)$. In Section~\ref{sec:simu} we discussed approximate simulation where one simulates a compound Poisson process containing the large jumps. To improve the convergence rate, the small jumps can be approximated with a suitable Brownian motion; see Remark~\ref{rem:small_jumps}. Previously a jump was considered `large' if its sup norm was above some threshold $\epsilon$. Instead we will now fix a non-empty subset $U\subseteq V$ and say that a jump is large if it exceeds $\epsilon$ in absolute value in at least one component of $U$. Note that $U=V$ corresponds to the previous notion of large jumps.

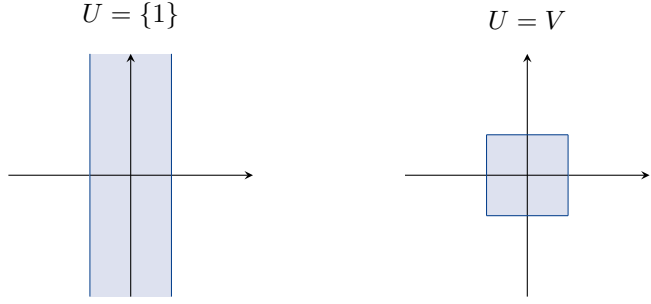
\begin{figure}[tb]
\begin{tikzpicture}
\begin{axis}[name=plot1,axis lines=center, width=6.2cm,height=4.8cm, ymin=-3,ymax=3,xmin=-3,xmax=3,ticks=none, title={$U=\Set{1}$},unit vector ratio*=1 1 1]
\draw[white,name path=A] (axis cs:-1,-3) to (axis cs:-1,3);
\draw[white,name path=B] (axis cs:1,-3) to (axis cs:1,3);
\addplot[nodecolor] fill between[of=A and B];
\draw[black] (axis cs:-1,0) to (axis cs:1,0);
\draw[black,-stealth] (axis cs:0,-3) to (axis cs:0,3);
\draw[AUblue] (axis cs:-1,-3) to (axis cs:-1,3);
\draw[AUblue] (axis cs:1,-3) to (axis cs:1,3);
\end{axis}

\begin{axis}[at={($ (2cm,0cm) + (plot1.south east) $)},axis lines=center, width=6.2cm,height=4.8cm, ymin=-3,ymax=3,xmin=-3,xmax=3,ticks=none,title={$U=V$},unit vector ratio*=1 1 1]
\draw[white,name path=A] (axis cs:-1,-1) to (axis cs:-1,1);
\draw[white,name path=B] (axis cs:1,-1) to (axis cs:1,1);
\addplot[nodecolor] fill between[of=A and B];
\draw[black] (axis cs:-1,0) to (axis cs:1,0);
\draw[black] (axis cs:0,-1) to (axis cs:0,1);
\draw[AUblue] (axis cs:-1,-1) to (axis cs:-1,1);
\draw[AUblue] (axis cs:1,-1) to (axis cs:1,1);
\draw[AUblue] (axis cs:-1,-1) to (axis cs:1,-1);
\draw[AUblue] (axis cs:-1,1) to (axis cs:1,1);
\end{axis}
\end{tikzpicture}
\caption{Jumps in the blue area are considered small.}
\label{fig:small_large}
\end{figure}

We may represent $\mathbf{X}$ as an independent sum $\mathbf{X}=\mathbf{X}\ph{(\epsilon)}+\mathbf{W}\ph{(\epsilon)}+\mathbf{\gamma}\ph{(\epsilon)}$, where $\mathbf{\gamma}\ph{(\epsilon)}$ is a linear drift, $\mathbf{W}\ph{(\epsilon)}$ is a martingale and a L\'evy process with L\'evy measure given by $\Lambda$ restricted to the set $\Set{\abs{x_u}\leq \epsilon\,\forall u\in U}$, and $\mathbf{X}\ph{(\epsilon)}$ is a compound Poisson process with jumps exceeding~$\epsilon$ in absolute value in at least one component of~$U$.
 
In the following we assume that
\begin{equation}\label{eq:L2}
\int_{\Set{\abs{x_u}\leq 1\,\forall u\in U}} x_i^2\,\Lambda(\dd x)<\infty\For*{for all $i\notin U$,}
\end{equation} 
and note that this property readily extends to all $i\in V$ from the properties of $\Lambda$. 
It is important to note that this assumption implies that
\begin{equation*}
\Lambda(\Set{x_U=0_U})=0,
\end{equation*}
because the L\'evy measure $\Lambda^0_{V\setminus U}$ is $-\a$-homogeneous and integrates $x_{V\setminus U}\mapsto x_i^2$ for all $i\notin U$.
Now we may define the matrix 
\begin{equation}\label{eq:Sigma}
\Sigma = \int_{\Set{\abs{x_u}\leq 1\,\forall u\in U}} xx^\top\,\Lambda(\dd x),
\end{equation}
where existence of the integral on the right-hand side is ensured by the assumption above.

\begin{lemma}\label{lem:small_jump_conv}
Assume \eqref{eq:L2}. Then with $a_\ep=\ep^{1-\a/2}$ it holds that
\begin{equation}\label{eq:limit}
\mathbf{W}\ph{(\epsilon)}/a_{\ep}\convd \mathbf{W}\For*{as $\epsilon\downarrow 0$,}
\end{equation}
where $\mathbf{W}$ is a drift-less Brownian motion with covariance matrix~$\Sigma$ defined in~\eqref{eq:Sigma}.
\end{lemma}

The next step is to establish conditions such that conditional independence for $\Lambda$ transfers to the classical conditional independence for the limiting Brownian motion $\mathbf{W}$ in Lemma~\ref{lem:small_jump_conv}.
We start with the simpler case of independence.

\begin{lemma}
Assume that $A \perp B\,[\Lambda]$ for a partition $A,B\subseteq V$, and that~\eqref{eq:L2} is satisfied. Then the Brownian motions $\mathbf{W}_A$ and $\mathbf{W}_B$ are independent.
\end{lemma}
\begin{proof}
Recall that \ref{eq:explosion_1} is satisfied because the process is stable. According to \citet[Proposition~1]{eng_iva_kir} we have that $\Lambda(\Set{x_i\neq 0,x_j\neq 0})=0$ for $i\in A,j\in B$. Hence, $\Sigma_{ij}=0$ and the result follows.
\end{proof}

The case of conditional independence is much more subtle, and it requires several strong assumptions.

\begin{theorem}\label{thm:BM}
Assume that $U$ satisfies \eqref{eq:L2}, and that $A \perp B \mid C\,[\Lambda]$ for a partition $A,B,C\subseteq V$. Moreover, assume that
\begin{itemize}
\item $\abs{C}=1$.
\item $\Lambda(\Set{x_C<0})=0$.
\item $U$ is a subset of either $A\cup C$ or $B\cup C$.
\item $\Sigma$ defined in~\eqref{eq:Sigma} is invertible.
\end{itemize}
Then $(\Sigma)^{-1}_{ij}=0$ for all $i\neq j$, $i\in A,j\in B$, and hence $W_A(1) \Perp W_B(1)\mid W_C(1)$.
\end{theorem}

\subsection{Proof of Lemma~\ref{lem:small_jump_conv}}

Letting $\Lambda\ph{(\epsilon)}$ be the L\'evy measure of the pre-limit process $\mathbf{W}\ph{(\epsilon)} /a_\ep$ we find
\begin{align*}
\int xx^\top\,\Lambda\ph{(\epsilon)}(\dd x)&=a_\ep^{-2}\int_{\Set{\abs{x_u}\leq \ep\,\forall u \in U}} xx^\top\,\Lambda(\dd x) \\
&=\int_{\Set{\abs{x_u}\leq 1\,\forall u \in U}} xx^\top\,\Lambda(\ep\dd x)\ep^{\a} \\
&=\Sigma,
\end{align*}
where the last equality follows from homogeneity of $\Lambda$.

Next we show that $\Lambda\ph{(\epsilon)}\vaguely*0$ on $\overline{\R^d}\setminus\Set{0}$. That is, for every $h>0$ and $i\in V$ we have the convergence $\Lambda\ph{(\epsilon)}(\Set{\abs{x_i}\geq h})\to 0$ as $\ep\downarrow 0$. Using homogeneity again we find that
\begin{align*}
\Lambda\ph{(\epsilon)}(\Set{\abs{x_i}\geq h})&=\Lambda(\Set{\abs{x_u}\leq\epsilon\,\forall u\in U,\abs{x_i}\geq ha_\epsilon}) \\
&=\epsilon^{-\alpha}\Lambda(\Set{\abs{x_u}\leq1\,\forall u\in U,\abs{x_i}\geq h\epsilon^{-\alpha/2}}).
\end{align*}
On the latter set we have $\epsilon^{-\alpha}\leq h^{-2}x_i^2$ and so we find that
\begin{equation*}
\Lambda\ph{(\epsilon)}(\Set{\abs{x_i}\geq h})\leq h^{-2}\int_{\abs{x_u}\leq 1\,\forall u\in U, \abs{x_i}\geq h\ep^{-\a/2}}x_i^2\,\Lambda(\dd x)\to0,
\end{equation*}
where we used dominated convergence with dominating function $x\mapsto x_i^2\ind{\Set{\abs{x_u}\leq1\,\forall u\in U}}$.

According to~\citet[Theorems~7.7 and~16.14]{kallenberg3} it is left to show
\begin{equation*}
\int_{\Set{\norm{x}\leq 1}} xx^\top\,\Lambda\ph{(\epsilon)}(\dd x)\to\Sigma
\qtq{and}
\gamma_\ep\to 0,
\end{equation*}
where $\gamma_\ep$ is obtained from the martingale requirement, i.e.\ $\gamma_\ep +\int_{\Set{\norm{x}>1}} x\,\Lambda\ph{(\epsilon)}(\dd x)=0$. Both are implied by
\begin{equation*}
\int_{\Set{\norm{x}>1}} x_i^2\,\Lambda\ph{(\epsilon)}(\dd x)=\int_{\Set{\abs{x_u}\leq 1\,\forall u \in U, \norm{x}\geq \ep^{-\a/2}}} x_i^2\,\Lambda(\dd x)\to 0,
\end{equation*}
derived using the same calculations as above. For the second we also employ the vague convergence to the zero measure.

\subsection{Proof of Theorem~\ref{thm:BM}}

We start by proving two auxiliary result.
Since $\mathbf{X}$ is a stable L\'evy process, we can use the homogeneity of $\Lambda$ to derive the existence of certain random vectors which will become essential in describing the structure of the jumps.
For a non-empty subset $C\subseteq V$ there is a $\Lambda_C$-unique probability kernel $\nu_C\colon\E^C\times\B(\R^{V\setminus C})\to[0,1]$ such that for any $R\in\mathcal{R}(\Lambda)$ with $0_C\notin R_C$ we have that
\begin{equation*}
\Lambda(R)=\int_{R_C}\nu_C(x_C,R_{V\setminus C})\,\Lambda_C(\dd x_C);
\end{equation*}
see \citet[Lemma~4.3]{eng_iva_kir}.
Furthermore, for any $c\in C$ and $\epsilon>0$ with $\Lambda(R_{c,\epsilon})>0$ there is the identity
\begin{equation}\label{eq:kernel_conditioning}
\nu_C(x_C,R_{V\setminus C})=\p_{c,\epsilon}(Y_{V\setminus C}\in R_{V\setminus C}\mid Y_C=x_C)
\end{equation}
for $\Lambda_C$-almost all $x_C$ with $\abs{x_c}\geq\epsilon$.
We will focus on the case where $C$ contains just one element $c$. 

\begin{lemma}\label{lem:mu_measures}
For any $c\in V$ there exist a $d$-dimensional random vector $\xi\ph{(c,+)}$ such that $\xi\ph{(c,+)}_c= 1$ almost surely, and for $\Lambda_c$-almost all $h >0$
\begin{equation*}
\nu_{\Set{c}}(h,\cdot)=\Prob{ h\xi\ph{(c,+)}_{V\setminus\Set{c}}\in\cdot\,}\For*{for $h>0$.}
\end{equation*}
\end{lemma}

\begin{proof}

For any Borel set $E_{V\setminus\Set{c}}\subseteq\R^{V\setminus\Set{c}}$ we introduce an auxiliary set $\Delta^+_{E_{V\setminus \Set{c}}}=\Set{(h,hx_{V\setminus\Set{c}})\given h\geq1,x_{V\setminus\Set{c}}\in E_{V\setminus\Set{c}}}$ as illustrated in Figure~\ref{fig:Delta_set_illustration} below. We can then define a probability measure $\mu_c^+$ by
\begin{equation*}
\mu^+_c(E_{V\setminus\Set{c}})=\Lambda(\Delta^+_{E_{V\setminus\Set{c}}})/\Lambda_c([1,\infty)),\For{$E_{V\setminus\Set{c}}\in\B(\R^{V\setminus{\Set{c}}})$}.
\end{equation*}
Now, let $\xi\ph{(c,+)}$ be a $d$-dimensional random vector such that $\xi\ph{(c,+)}_{V\setminus\Set{c}}\sim\mu_c^+$ and $\xi\ph{(c,+)}_c=1$.
For any $\epsilon>0$ and $h>0$ we have
\begin{align*}
\int_{[\epsilon,\infty)}\nu_{\Set{c}}(h,hE_{V\setminus\Set{c}})\,\Lambda_c(\dd h)&=\Lambda(\epsilon\Delta^+_{E_{V\setminus\Set{c}}}) \\
&=\epsilon^{-\alpha}\Lambda(\Delta^+_{R_{A\cup B}}) \\
&=\epsilon^{-\alpha}\Lambda_c([1,\infty))\mu^+_c(E_{V\setminus\Set{c}}) \\
&=\Lambda_c([\epsilon,\infty))\mu^+_c(E_{V\setminus\Set{c}}).
\end{align*}
Hence, $\nu_{\Set{c}}(h,hE_{V\setminus\Set{c}})=\mu_c^+(E_{V\setminus{c}})$ for $\Lambda_c$-almost all $h\geq\epsilon$. Using standard arguments we may extend this to
\begin{equation*}
\nu_{\Set{c}}(h,hE_{V\setminus\Set{c}})=\mu_c^+(E_{V\setminus\Set{c}})\For*{for all $E_{V\setminus\Set{c}}\in\B(\R^{V\setminus\Set{c}})$}
\end{equation*}
for $\Lambda_c$-almost all $h>0$. Hence,
\begin{equation*}
\nu_{\Set{c}}(h,E_{V\setminus\Set{c}})=\mu_c^+(h^{-1}E_{V\setminus\Set{c}})=\Prob{h\xi\ph{(c,+)}\in E_{V\setminus\Set{c}}}\For*{for all $E_{V\setminus\Set{c}}\in\B(\R^{V\setminus\Set{c}})$}
\end{equation*}
for $\Lambda_c$-almost all $h>0$.

\begin{figure}[tb]
\begin{tikzpicture}
\begin{axis}[axis lines=left,tick pos=left, width=.7\linewidth,height=6cm, ymin=0,ymax=5,ytick={1},xmin=0,xmax=10,xtick={1,2},xticklabels={\(E_{V\setminus\Set{c}}\),},x tick label style={anchor=north west},tick align=outside, major tick length=2.75pt,every tick/.style={black},ylabel=\(h\),ylabel style={rotate=-90}]
\draw[dashed,color=black!40] (axis cs:0,1) to (axis cs:10,1);
\addplot+[AUblue,domain=1:2,mark=none] {1};
\addplot+[name path=A,AUblue,domain=1:5,mark=none] {x};
\addplot+[name path=B,AUblue,domain=2:10,mark=none] {x/2};
\addplot[nodecolor] fill between[of=A and B];
\node[color=AUblue] at (axis cs: 5,3.5) {$\Delta^+_{E_{V\setminus\Set{c}}}$};
\end{axis}
\end{tikzpicture}
\caption{Illustration of the set $\Delta^+_{E_{V\setminus\Set{c}}}$.}
\label{fig:Delta_set_illustration}
\end{figure}
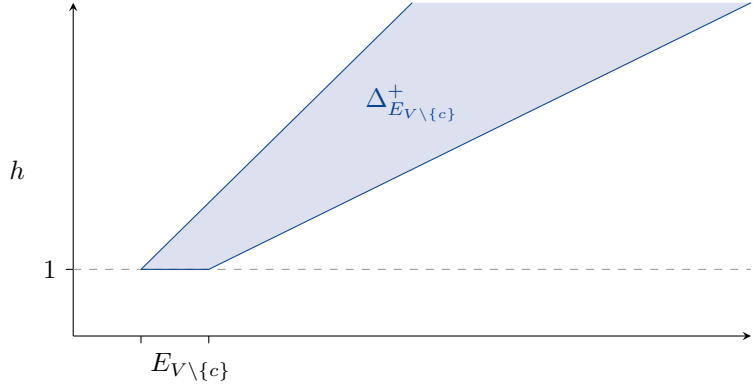

\end{proof}

\begin{lemma}\label{lem:matrix}
Consider a positive semi-definite block matrix of the form
\[M=\begin{pmatrix} 
\delta & \delta m_1^\top & m_2^\top\\
\delta m_1 & \delta M_{11} &m_1 m_2^\top\\
m_2 & m_2m_1^\top &M_{22}\\
\end{pmatrix},\]
where $\delta$ is a scalar, $M_{11},M_{22}$ are square matrices, and $m_1,m_2$ are vectors of appropriate dimensions. If $M$ is positive-definite (i.e., invertible) then the $(2,3)$ and $(3,2)$ blocks of  $M^{-1}$ are zero.
\end{lemma}

\begin{proof}
For a positive definite block matrix $X = \begin{pmatrix}P &Q\\ R & S\end{pmatrix} $ with square blocks $P,S$ it is well known that

\begin{equation*}
X^{-1}=\begin{pmatrix}
P^{-1}+P^{-1}QYRP^{-1} &-P^{-1}QY \\
-YRP^{-1} & Y
\end{pmatrix}, \qquad Y=(S-RP^{-1}Q)^{-1}.
\end{equation*}
Firstly, we have
\begin{equation*}
\begin{pmatrix} 
1 & m_1^\top\\
m_1 & M_{11} 
\end{pmatrix}^{-1}=\begin{pmatrix} 
\cdot & \cdot\\
-(M_{11}-m_1m_1^\top)^{-1}m_1 & (M_{11}-m_1m_1^\top)^{-1}
\end{pmatrix},
\end{equation*}
where the first row is not important to us.
Secondly, we partition $M$ so that $S=M_{22}$ and obtain the following representation of the  $(2,3)$ block of $M^{-1}$:
\begin{equation*}-\delta^{-1}[-(M_{11}-m_1m_1^\top)^{-1}m_1, (M_{11}-m_1m_1^\top)^{-1}]\begin{pmatrix} 
m_2^\top\\
m_1m_2^\top
\end{pmatrix}Y\\
=0.
\end{equation*}
\end{proof}

\begin{proof}[Proof of Theorem~\ref{thm:BM}]
Without loss of generality we may assume that $C=\{1\}$, the elements of~$A$ are smaller than the elements of~$B$, and $U\subseteq B\cup\{1\}$. According to Lemma~\ref{lem:matrix} it is sufficient to establish that $\Sigma $ has the form of~$M$. 

The marginal measure is given by
\begin{equation*}
\Lambda_1(\dd x_1)=m_1\alpha x_1^{-\a-1}\dd x_1,\For{$x_1>0$,}
\end{equation*}
where $m_1=\Lambda(\Set{x_1>1})$.
For $x_1>0$ consider the kernel $\nu_{\Set{1}}(x_1,\cdot)$. Recall that Lemma~\ref{lem:mu_measures} gives the existence of a $d$-dimensional random vector $\xi\ph{(1)}$ such that $\xi\ph{(1)}_1=1$ a.s.\ and
\begin{equation*}
\nu_{\Set{1}}(x_1,\cdot)=\p(x_1 \xi\ph{(1)}_{V\setminus\Set{1}}\in \cdot\,)
\end{equation*}
for $\Lambda_1$-almost all $x_1>0$.

For $i\in A\cup\Set{1}$ and $j\in B\cup\Set{1}$ we may compute
\begin{align*}
\Sigma_{ij}&=\int_{\Set{\abs{x_u}\leq 1\,\forall u\in U}} x_ix_j\,\Lambda(\dd x) \\
&=m_1\alpha\e\left[\int_{\Set{x_1>0,\abs{x_1\xi\ph{(1)}_u}\leq 1\,\forall u\in U}} (x_1\xi\ph{(1)}_i)(x_1\xi\ph{(1)}_j)  x_1^{-\a-1}\idd x_1\right] \\
&=\frac{m_1\alpha}{2-\a} \e \left[\xi\ph{(1)}_i\xi\ph{(1)}_j\min_{u\in U}\abs{\xi\ph{(1)}_u}^{\a-2}\right].
\end{align*}
For the second equality we used that $x_1=0$ has no contribution to~$\Sigma$. If $i=1$ or $j=1$ this is obvious, and if $i\in A,j\in B$ we recall that $A\perp B\mid \Set{1}\,[\Lambda]$ implies $\Lambda^0_{V\setminus\{1\}}(\Set{x_i\neq 0,x_j\neq 0})=0$, see \citet[Lemma~3.3 and~Proposition~5.1]{eng_iva_kir}.

According to \citet[Theorem~4.4]{eng_iva_kir} the conditional independence $A \perp B \mid \{1\}\,[\Lambda]$ implies $\xi\ph{(1)}_A \Perp \xi\ph{(1)}_B$.
Letting $\delta = \frac{m_1\alpha}{2-\a} \e \left[\min_{u\in U}\abs{\xi\ph{(1)}_u}^{\a-2}\right]$ we find that 
\begin{align*}
\Sigma_{11}&=\delta, \\
\Sigma_{1i}&=\delta\Mean{\xi\ph{(1)}_i}, \\
\Sigma_{1j}&=\frac{m_1\alpha}{2-\a} \e \Big[\xi_j\min_{u\in U}\abs{\xi\ph{(1)}_u}^{\a-2}\Big], \\
\Sigma_{ij}&=\frac{m_1\alpha}{2-\a} \Mean{\xi\ph{(1)}_i}\e\Big[\xi_j\min_{u\in U}\abs{\xi\ph{(1)}_u}^{\a-2}\Big].
\end{align*}
This shows that $\Sigma$ has the form of $M$ in Lemma~\ref{lem:matrix}.

\end{proof}

\bibliographystyle{abbrvnat}
\bibliography{ref}

\end{document}